\documentclass[tbtags,reqno]{amsart}
\usepackage{graphicx,color}
\usepackage{amsmath,amsfonts,amstext,amsbsy,amsopn,amsthm,amscd,amsxtra,dsfont,amssymb,pb-diagram}

\usepackage{amsmath,amsthm,amsopn,amsfonts,amssymb,epsfig,enumerate}
\usepackage[all]{xy}
\usepackage[active]{srcltx}

\edef\savecatcodeat{\the\catcode`@} \catcode`\@=11

\def\tb@ifSpecChars#1#2{#1}
\def\tb@ifNoSpecChars#1#2{#2}

\def\tableau{%
  \bgroup
  \@ifstar{\let\Tif\tb@ifNoSpecChars\tb@tableauB}
          {\let\Tif\tb@ifSpecChars\tb@tableauB}}

\def\tb@tableauB{
  \@ifnextchar[{\tb@tableauC}{\tb@tableauC[]}}

\def\tb@tableauC[#1]{\hbox\bgroup%
    \let\\=\cr
    \def\bl{\global\let\tbcellF\tb@cellNF}%
    \def\tf{\global\let\tbcellF\tb@cellH}
    \dimen2=\ht\strutbox \advance\dimen2 by\dp\strutbox%
    \ifx\baselinestretch\undefined\relax%
    \else%
       \dimen0=100sp \dimen0=\baselinestretch\dimen0%
       \dimen2=100\dimen2 \divide\dimen2 by\dimen0%
    \fi%
    \let\tpos\tb@vcenter
    \tb@initYoung
    \tb@options#1\eoo
    \let\arrow\tb@arrow%
    \dimen0=\Tscale\dimen2%
    \dimen1=\dimen0 \advance\dimen1 by \tb@fframe%
    \lineskip=0pt\baselineskip=0pt
    \def\tb@nothing{}%
    \def\endcellno{$\rss\egroup\bss\egroup}
    \def\endcell{\endcellno\kern-\dimen0}
    \def\begincell{\vbox to\dimen0\bgroup\vss\hbox to\dimen0\bgroup\hss$}%
    \let\overlay\tb@overlay%
    \let\fl\tb@fl%
    \let\lss\hss\let\rss\hss\let\tss\vss\let\bss\vss
    \def\mkcell##1{
        \let\tbcellF\tb@cellD
        \def\tb@cellarg{##1}
        \ifx\tb@cellarg\tb@nothing\let\tb@cellarg\tb@cellE\fi%
        \begincell\tb@cellarg\endcellno
        \tbcellF}
    \let\savecellF\tbcellF
     \Tif{\catcode`,=4\catcode`|=\active}{}\tb@tableauD}%

\let\tb@savetableauD\tableauD
{
    \catcode`|=\active \catcode`*=\active \catcode`~=\active%
    \catcode`@=\active
\gdef\tableauD#1{%
  \Tif{
    \mathcode`|="8000 \mathcode`*="8000%
    \mathcode`~="8000 \mathcode`@="8000%
    \def@{\bullet}%
    \let|\cr
    \let*\tf
    \let~\sk
  }{}%
  \tpos{\tabskip=0pt\halign{&\mkcell{##}\cr#1\crcr}}%
  \global\let\tbcellF\savecellF
  \egroup
  \egroup}
}
\let\tb@tableauD\tableauD
\let\tableauD\tb@savetableauD
\let\tb@savetableauD\undefined

\def\tb@options#1{\ifx#1\eoo\relax\else\tb@option#1\expandafter\tb@options\fi}

\def\tb@option#1{%
  \if#1t\let\tpos\tb@vtop\fi
  \if#1c\let\tpos\tb@vcenter\fi
  \if#1b\let\tpos\vbox\fi
  \if#1F\tb@initFerrers\fi
  \if#1Y\tb@initYoung\fi
  \if#1s\tb@initSmall\fi
  \if#1m\tb@initMedium\fi
  \if#1l\tb@initLarge\fi
  \if#1p\tb@initPartition\fi
  \if#1a\tb@initArrow\fi
}

\def\tb@vcenter#1{\ifmmode\vcenter{#1}\else$\vcenter{#1}$\fi}

\def\tb@vtop#1{\hbox{\raise\ht\strutbox\hbox{\lower\dimen0\vtop{#1}}}}

\def\tb@initPartition{\def\Tscale{.3}}
\def\tb@initSmall{\def\Tscale{1}}
\def\tb@initMedium{\def\Tscale{2}}
\def\tb@initLarge{\def\Tscale{3}}

\def\tb@initArrow{\dimen2=1.25em}

\def\tb@initYoung{%
  \def\tb@cellE{}
  \let\tb@cellD\tb@cellN
  \def\sk{\global\let\tbcellF\tb@cellNF}}
\def\tb@initFerrers{%
  \def\tb@cellE{\bullet}
  \let\tb@cellD\tb@cellNF
  \def\sk{\bullet}}

\tb@initMedium

\def\tb@sframe#1{%
  \vbox to0pt{
    \vss
    \hbox to0pt{%
      \hss
      \vbox to\dimen1{
        \hrule depth #1 height0pt
        \vss
        \hbox to\dimen1{
          \vrule width #1 height\dimen1
          \hss
          \vrule width #1
          }%
        \vss
        \hrule height #1 depth 0in
        }%
      \kern-\tb@hframe
      }%
    \kern-\tb@hframe}}

\def\tb@hframe{.2pt}\def\tb@fframe{.4pt}\def\tb@bframe{1.2pt}
\def\tb@cellH{\tb@sframe{\tb@bframe}}       
\def\tb@cellNF{}                            
\def\tb@cellN{\tb@sframe{\tb@fframe}}       
\let\tbcellF\tb@cellN                       

\def\tb@rpad{1pt}
\def\tb@lpad{1pt}
\def\tb@tpad{1.8pt}
\def\tb@bpad{1.8pt}

\def\tb@overlay{\endcell\@ifnextchar[{\tb@overlaya}{\begincell}}
\def\tb@overlaya[#1]{\vbox to\dimen0\bgroup%
  \tb@overlayoptions#1\eoo%
  \tss\hbox to\dimen0\bgroup\lss}
\def\tb@overlayoptions#1{\ifx#1\eoo\relax\else\tb@overlayoption#1\expandafter\tb@overlayoptions\fi}

\def\tb@overlayoption#1{
  \if#1t\def\tss{\vskip\tb@tpad}\let\bss\vss\fi
  \if#1c\let\tss\vss\let\bss\vss\fi
  \if#1b\def\bss{\vskip\tb@bpad}\let\tss\vss\fi
  \if#1l\def\lss{\hskip\tb@lpad}\let\rss\hss\fi
  \if#1m\let\lss\hss\let\rss\hss\fi
  \if#1r\def\rss{\hskip\tb@rpad}\let\lss\hss\fi
}

\def\tb@fl{\endcell\begincell\vrule depth 0pt width \dimen0 height \dimen0 \endcell\begincell}

\@ifundefined{diagram}{}{
\def\tb@arrowpad{.5}

\newoptcommand{\tb@arrow}{\@ne}[2]{%
  \endcell
   \begingroup%
   \let\dg@getnodesize\tb@getnodesize
   \dg@USERSIZE=#1\relax%
   \ifnum\dg@USERSIZE<\@ne \dg@USERSIZE=\@ne \fi%
   \dg@parse{#2}%
   \dg@label{\tb@draw{#1}{#2}}}

\def\tb@getnodesize#1#2#3#4#5{\dimen3=\tb@arrowpad\dimen2 #4=\dimen3 #5=\dimen3\relax}
\def\tb@getnodesize#1#2#3#4#5{\ifnum#2=0\ifnum#3=0\tb@getnodesizetail{#4}{#5}\else\tb@getnodesizehead{#4}{#5}\fi\else\tb@getnodesizehead{#4}{#5}\fi}
\def\tb@getnodesizetail#1#2{\dimen3=.5\dimen2 #1=\dimen3 #2=\dimen3}
\def\tb@getnodesizehead#1#2{\dimen3=.5\dimen2 #1=\dimen3 #2=\dimen3}

\def\tb@draw#1#2#3#4{%
        \dg@X=0\dg@Y=0\dg@XGRID=1\dg@YGRID=1\unitlength=.001\dimen0%
        \dg@LBLOFF=\dgLABELOFFSET \divide\dg@LBLOFF\unitlength%
        \dg@drawcalc
        \begincell
        \let\lams@arrow\tb@lams@arrow
        \begin{picture}(0,0)\begingroup\dg@draw{#1}{#2}{#3}{#4}\end{picture}%
        \endcell
        \endgroup
        \begincell}
}

\def\tb@lams@arrow#1#2{%
 \lams@firstx\z@\lams@firsty\z@
 \lams@lastx#1\relax\lams@lasty#2\relax
 \lams@center\z@
 \N@false\E@false\H@false\V@false
 \ifdim\lams@lastx>\z@\E@true\fi
 \ifdim\lams@lastx=\z@\V@true\fi
 \ifdim\lams@lasty>\z@\N@true\fi
 \ifdim\lams@lasty=\z@\H@true\fi
 \NESW@false
 \ifN@\ifE@\NESW@true\fi\else\ifE@\else\NESW@true\fi\fi
 \ifH@\else\ifV@\else
  \lams@slope
  \ifnum\lams@tani>\lams@tanii
   \lams@ht\ten@\p@\lams@wd\ten@\p@
   \multiply\lams@wd\lams@tanii\divide\lams@wd\lams@tani
  \else
   \lams@wd\ten@\p@\lams@ht\ten@\p@
   \divide\lams@ht\lams@tanii\multiply\lams@ht\lams@tani
  \fi
 \fi\fi
 \ifH@  \lams@harrow
 \else\ifV@ \lams@varrow
 \else \lams@darrow
 \fi\fi
}

\catcode`\@=\savecatcodeat
\let\savecatcodeat\undefined

\numberwithin{equation}{section}

\makeatletter
\renewcommand{\subsubsection}{\@startsection
{subsubsection} {3} {0mm} {\baselineskip} {-0.5\baselineskip} {\normalfont\normalsize\bfseries}} \makeatother

\newtheorem{theorem}{Theorem}
\newtheorem{lemma}[theorem]{Lemma}
\newtheorem{proposition}[theorem]{Proposition}
\newtheorem{example}[theorem]{Example}
\newtheorem{conjecture}[theorem]{Conjecture}

\newtheorem{definition}[theorem]{Definition}

\newtheorem{remark}[theorem]{Remark}

\def\charge{ {\rm {ch}}}
\def\cocharge{ {\rm {coch}}}

\def\la{{\lambda}}

\def\cal L{{\mathcal L}}

\def\nup {n^\uparrow}
\def\ndown {n^\downarrow}
\def\npup {\bar n^\uparrow}
\def\npdown {\bar n^\downarrow}
\def\nplusup {n_+^\uparrow}
\def\nplusdown {n_+^\downarrow}
\def\npluspup {\bar n_+^\uparrow}
\def\npluspdown {\bar n_+^\downarrow}
\newcommand{\kshapeposet}{poset of $k$-shapes}

\newcommand{\bdy}{\partial}

\newcommand{\bp}{\mathbf{p}}
\newcommand{\bq}{\mathbf{q}}

\newcommand{\ch}{\mathrm{ch}}
\newcommand{\change}{\Delta}

\newcommand{\col}{\mathrm{col}}

\newcommand{\Core}{\mathcal{C}}
\newcommand{\cs}{\mathrm{cs}}
\newcommand{\diag}{d}

\newcommand{\Int}{\mathrm{Int}}
\newcommand{\Ksh}{\Pi}

\newcommand{\Path}{\mathcal{P}}
\newcommand{\Patheq}{\overline{\Path}}

\newcommand{\row}{\mathrm{row}}
\newcommand{\rs}{\mathrm{rs}}

\newcommand{\tM}{\tilde{M}}
\newcommand{\tm}{\tilde{m}}

\newcommand{\Z}{\mathbb{Z}}

\def \part {\vdash}

\begin{document}

\title[Charge on tableaux and the poset of $k$-shapes]
{Charge on tableaux and the poset of $k$-shapes}

\author{Luc Lapointe }
\thanks{Research partially supported by FONDECYT (Fondo Nacional de Desarrollo Cient\'{\i}fico y
Tecnol\'ogico de Chile) grant \#1090016, by CONICYT (Comisi\'on Nacional de Investigaci\'on Cient\'ifica y Tecnol\'ogica de Chile) via
the ``proyecto anillo ACT56'', and by NSF grant DMS-0652641.}
\address{Instituto de Matem\'atica y F\'{\i}sica, Universidad de Talca,
Casilla 747, Talca, Chile} \email{lapointe-at-inst-mat.utalca.cl}

\author{Mar\'{\i}a Elena Pinto }
\address{Instituto de Matem\'atica y F\'{\i}sica, Universidad de Talca, Casilla 747, Talca,
Chile} \email{mepinto-at-inst-mat.utalca.cl}


\begin{abstract}   A poset on a certain class of partitions known as $k$-shapes was introduced in \cite{LLMS2} to provide a combinatorial rule for the expansion of
a  $k-1$-Schur functions into $k$-Schur functions at $t=1$.  The main ingredient in this construction was a bijection, which we call the weak bijection, that associates to a $k$-tableau a pair made out of a $k-1$-tableau and a path in the poset of $k$-shapes. We define here a concept of charge on $k$-tableaux (which conjecturally gives a combinatorial interpretation for the expansion coefficients of Hall-Littlewood polynomials into $k$-Schur functions),
and show that it is compatible in the standard case with the weak bijection.
In particular, we obtain that the usual charge of a standard tableau of size $n$
is equal to the sum of the charges of its corresponding paths in the poset of $k$-shapes, for
$k=2,3,\dots,n$.
\end{abstract}


\maketitle

\section{Introduction}

For each integer $k \geq 1$, a family of symmetric functions
$s_\mu^{(k)}(x;t)$ (now called $k$-Schur functions)
were introduced
in \cite{LLM} in connection with
Macdonald polynomials.
To be more precise, computer evidence suggested, for each positive integer $k$,
the existence of
a family of symmetric polynomials defined by certain sets of
tableaux $\mathcal A^{(k)}_\mu$ as:
\begin{equation} \label{eqatom}
s_{\mu}^{(k)}(x;t) = \sum_{T\in \mathcal A^{(k)}_\mu}
t^{{\rm charge}(T)} \, s_{{\rm shape}(T)}(x)
\end{equation}
with the property that
any (plethystically modified) Macdonald polynomial $H_{\lambda}(x;q,t)$
indexed by a partition $\lambda$ whose first part is not larger than $k$ (a $k$-bounded partition),
can be decomposed as:
\begin{equation} \label{eqMac}
H_{\lambda}(x;q,t) = \sum_{\mu;\mu_1\leq k}
K_{\mu \lambda}^{(k)}(q,t) \, s_{\mu}^{(k)}(x;t) \, , \qquad
K_{\mu \lambda}^{(k)}(q,t) \in \mathbb N[q,t] \, .
\end{equation}
Moreover, given that in this setting $s_\mu^{(k)}(x;t)=s_\mu(x)$
for large $k$, the coefficients $K_{\mu \lambda}^{(k)}(q,t)$ reduce to the usual
$q,t$-Kostka coefficients  when $k$ is large enough \cite{H,M}.
 The study of the
$s_\lambda^{(k)}(x;t)$ led to several
conjecturally equivalent
characterizations \cite{LLM,LM:filtra,LLMS}, but for the purpose
of this article only the characterization of \cite{LLMS} will be relevant.
Note that from now on,
we will always index $k$-Schur functions by $k+1$-cores rather
than by $k$-bounded partitions (see for instance
\cite{LM:cores} for the connection between the two concepts).

It was shown that the $k$-Schur functions at $t=1$ (in the characterization \cite{LLMS})
provide the natural basis to work in the quantum cohomology of the
Grassmannian just as the Schur functions do for the usual cohomology
\cite{LM:QC}.
Another key development was T.~Lam's
proof \cite{Lam:kSchur} that the
Schubert basis of the homology of the affine Grassmannian is given
by the $k$-Schur functions, and that the Schubert basis of the
cohomology of the affine Grassmannian is given
by functions dual to the $k$-Schur functions, called dual $k$-Schur functions or
affine Schur functions \cite{Lam:affstan,LM:QC}.

Many combinatorial conjectures about $k$-Schur functions  have been
formulated, but the conjecture especially relevant to this article is
that the $k$-Schur functions expand positively into $k'$-Schur
functions for $k'>k$ \cite{LLM}:
\begin{equation}\label{surepos0}
s_\lambda^{(k)}(x;t) = \sum_{\mu}
b_{\lambda,\mu}^{(k \to k')}(t)\, s_\mu^{(k')}(x;t)\,,\qquad\text{for
$b_{\lambda,\mu}^{(k \to k')}(t)\in\Z_{\ge0}[t]$.}
\end{equation}
When $k'$ is large enough, this conjecture says that $k$-Schur functions are Schur positive (partial results in this direction have been obtained in \cite{AB}).  The conjecture also states in particular
that $(k-1)$-Schur functions expand positively
into $k$-Schur functions.
We refer to the $b_{\lambda,\mu}^{(k-1,k)}(t)$
coefficients
as the
$k$-branching coefficients. We have defined in \cite{LLMS2}
a poset called the poset of
$k$-shapes, whose maximal elements are $(k+1)$-cores and whose
minimal elements are $k$-cores, and have a conjecture for the
$k$-branching coefficients as enumerating maximal chains in the
poset of $k$-shapes modulo an equivalence (see Section~\ref{missing}).
Using the duality between the ungraded $k$-Schur and dual $k$-Schur functions
\cite{LLMS}, this conjecture has been shown to be valid
when $t=1$ by proving that dual $k'$-Schur functions expand positively into dual $k$-Schur functions for $k'>k$  \cite{LLMS2}.

One of the obstructions to the generalization of the
the results of \cite{LLMS2} to a generic value
of $t$ was the lack of an explicit definition for the
graded version of the dual
$k$-Schur functions\footnote{Another version of the graded dual
$k$-Schur functions can be obtained from the $k$-Schur functions
by duality with respect to the Hall-Littlewood scalar product
(see \cite{LLMSSZ}).  This version is however not monomial positive.}.
  In this article we introduce
a charge on $k$-tableaux that provides such a definition.
To be more precise,
the dual $k$-Schur functions
are the generating function of certain combinatorial objects
called $k$-tableaux.
We define the graded version of the dual $k$-Schur functions
as\footnote{A different definition of graded dual $k$-Schur functions has been proposed in \cite{DM}.}
\begin{equation}
{\mathfrak S}_{\lambda}^{(k)}(x;t) = \sum_{Q} t^{{\charge} (Q)} x^Q
\end{equation}
where the sum is over all $k$-tableaux $Q$ of shape $\lambda$
(where for simplicity the dual $k$-Schur function is indexed by a $k+1$-core),
and where $\charge (Q)$ is a certain generalization (see Section~\ref{Schargektab})
of the charge
of a tableau \cite{LScharge}.   We conjecture that the charge also provides
the $k$-Schur expansion of a Hall-Littlewood polynomial indexed by a $k$-bounded partition
\begin{equation} \label{HLkschur0}
H_{\lambda}(x;t) = \sum_{Q^{(k)}} t^{{\rm ch}(Q^{(k)})} s_{{{\rm sh} (Q^{(k)})}}^{(k)}(x;t)
\end{equation}
where the sum is over all $k$-tableaux of weight $\lambda$, and where
${\rm sh}(Q^{(k)})$ is the shape of the $k$-tableau $Q^{(k)}$.
This formula,  which would prove the case $q=0$ of \eqref{eqMac},
generalizes a well-known result of
 Lascoux and Sch\"utzenberger providing
a $t$-statistic on tableaux for the Kostka-Foulkes polynomials \cite{LScharge}.

 We give evidence of the validity of the
definition of charge by
proving  that the concept of charge on $k$-tableaux
is compatible in the standard case (the
complications pertaining to the extension to the non-standard
case are discussed in the conclusion)
with the weak bijection introduced in \cite{LLMS2}.

\begin{theorem} \label{theo} The weak bijection in the standard case
\begin{equation}
\begin{split}
{\rm SWTab}_{\lambda}^k  & \longrightarrow  \bigsqcup_{\mu \in {\mathcal C}^k}
{\rm SWTab}_{\mu}^{k-1}
\times \overline{\mathcal P}^k(\lambda,\mu) \\
Q^{(k)}  & \longmapsto    (Q^{(k-1)},[\bp])
\end{split}
\end{equation}
where ${\rm SWTab}_{\lambda}^k$ is the set of standard $k$-tableau (or standard
weak tableau) of shape $\lambda$, and where
$[\bp]$ is a certain equivalence class of paths in the poset of $k$-shapes,
is such that
\begin{equation}
\charge(Q^{(k)} )= \charge(Q^{(k-1)}) + \charge(\bp)
\end{equation}
with $\charge(\bp)$ the charge of the path $\bp$ (see Definition~\ref{D:charge}).
\end{theorem}
As discussed in Section~\ref{missing},
Theorem~\ref{theo} is one of the key ingredients in our approach
to prove the $k$-Schur expansion \eqref{HLkschur0}
 of the Hall-Littlewood polynomials.

A simple consequence of the compatibility between the charge and the weak bijection is that the charge of a usual standard tableau $T$ of
of $n$ letters
is equal to the sum of the charges of its corresponding paths in the poset of $k$-shapes, for
$k=2,3,\dots,n$.
In effect,
iterating the weak bijection starting from $T$
\begin{equation}
T \mapsto (T^{(n-1)},[\bp_{n}]), \, \quad
 T^{(n-1)} \mapsto (T^{(n-2)},[\bp_{n-1}]), \quad
\dots \, , \, \quad T^{(2)} \mapsto (T^{(1)},[\bp_{2}])
\end{equation}
we obtain a bijection that puts in correspondence $T$ and
$(T^{(1)}, [\bp_n],[\bp_{n-1}],\dots,[\bp_2])$.  Given that
there is a unique standard $1$-tableau $T^{(1)}$ we have that
 $T$ is in correspondence with the equivalence of paths $([\bp_n],[\bp_{n-1}],\dots,[\bp_2])$.
Moreover, the charge of $T^{(1)}$ being 0,
the compatibility between the charge and the weak bijection
in the standard case implies
\begin{equation}
\charge(T)
= \charge(\bp_{n})+\cdots+ \charge(\bp_{2})
\end{equation}

Here is the outline of the article.
For the article to be self-contained,
a good deal of results and definitions of \cite{LLMS2} must be introduced
or specialized to the standard case.  These include for instance moves,
covers, the poset of $k$-shapes, $k$-shape tableaux,
the pushout algorithm, and the weak bijection.  This is essentially the content of Sections~\ref{Sprem}, \ref{SubS:poset},  \ref{Skshapetab}, \ref{SecPush}
and \ref{Sweakbij}.  The charge of a $k$-tableau is defined
in Section~\ref{Schargektab}.  In Section~\ref{missing} we describe
the general context of this work, such as the Schur positivity of $k$-Schur functions, the $k$-Schur positivity of Hall-Littlewood polynomials, and finally the
connection with the atoms of \cite{LLM}.  In Section~\ref{Schargekshape},
we define the charge and cocharge of a $k$-shape tableau and derive a
relation between the two concepts (Proposition~\ref{propchargecocharge}).  Section~\ref{Scompatibility}
contains the proof of
the main result of this article, namely the
compatiblity between (co)charge and the weak bijection in the standard case
(Theorem~\ref{theo}).
Finally, we discuss in the conclusion why the non-standard case is still out of reach and how fundamental is the problem of
defining a Lascoux-Sch\"utzenberger type action of the symmetric group on
$k$-shape tableaux.

\section{Preliminaries} \label{Sprem}

For a fixed positive integer $k$, the object central to our
study is a family of ``$k$-shape" partitions that contains
both $k$ and $k+1$-cores.  The formula for $k$-branching
coefficients counts paths in a poset on $k$-shapes.  As
with Young order,
we will define the order relation in terms of adding boxes
to a given vertex $\lambda$, but now the added boxes must
form a sequence of ``strings".  Here we introduce
$k$-shapes, strings, and moves -- the ingredients for
our poset.

\subsection{Partitions}

A partition $\lambda=(\lambda_1,\lambda_2,\dots)$ of degree $|\lambda|=\sum_i \lambda_i$ is a vector of non-negative integers such that
$\lambda_i \geq \lambda_{i+1}$ for $i=1,2,\dots$.
 The length $\ell(\lambda)$
of $\lambda$ is the number of non-zero entries of $\lambda$.
Each partition $\lambda$ has an associated Ferrers' diagram
with $\lambda_i$ lattice squares in the $i^{th}$ row,
from bottom to top (French notation).  For instance, if
$\lambda = (6,5,5,4,3,2,2,1,1,1)$, the associated Ferrers' diagram is
$${\tiny \tableau[scY]{|||,|,|,,|,,,|,,,,|,,,,|,,,,,}}$$

Any lattice square in the Ferrers diagram
is called a cell (or simply a square), where the cell $(i,j)$ is in the $i$th row and $j$th column of the diagram.  Given a cell $b=(i,j)$, we let
$\row(b)=i$ and $\col(b)=j$.
The conjugate $\lambda'$ of  a partition $\lambda$ is the partition whose diagram is
obtained by reflecting  the diagram of $\lambda$ about the main diagonal.
Given a cell $b=(i,j)$ in $\lambda$, we let
\begin{equation} \label{eqarms}
a_{\lambda}(b)=\lambda_i-j\, , \qquad {\rm and} \qquad l_{\lambda}(b)=\lambda_j'-i \,  .
\end{equation}
The quantities $a_{\lambda}(b)$ and $l_{\lambda}(b)$
are respectively called the arm-length and leg-length.
The hook-length of $b=(i,j) \in \la$ is then defined
by $h_\la(b) = a_\la(b)+l_\la(b)+1$.  A $p$-core is a  partition
without cells of hook-length equal to $p$. We
let $\Core^p$ be the set of $p$-cores.

We say that the diagram $\mu$ is contained in $\la$, denoted
$\mu\subseteq \la$, if $\mu_i\leq \la_i$ for all $i$.  We also let
$\lambda+\mu$ be the partitions whose entries are $(\lambda+\mu)_i=\lambda_i+\mu_i$, and $\lambda \cup \mu$ be the partition obtained by reordering the entries of the concatenation of $\lambda$ and $\mu$.  The dominance ordering on partitions is such that $\lambda \geq \mu$ iff $|\lambda|=|\mu|$ and $\lambda_1+\cdots+\lambda_i \geq \mu_1+\cdots+\mu_i$ for all $i$.

A cell $b$ is $\lambda$-addable (or $b$ is an addable corner of $\lambda$)
 if adding $b$ to $\lambda$ produces a partition
$\mu$.  Similarly, a cell of $b$ is $\lambda$-removable
(or $b$ is a removable corner of $\lambda$) if
removing $b$ from $\lambda$ leads to a partition $\mu$.
The diagonal index of $b=(i,j)$ is  $\diag(b)=j-i$.  From it
we define the {distance} between cells $x$ and $y$ as
$|\diag(x)-d(y)|$.

Let $D=\mu/\la$ be a skew shape, the difference of Ferrers diagrams
of partitions $\mu\supset\la$. Although such a set of cells may
be realized by different pairs of partitions,
unless specifically stated otherwise, we shall use the notation $\mu/\la$
with the fixed pair $\la\subset\mu$ in mind.
A {\it horizontal} (resp. {\it vertical}) strip is
a skew shape that contains at most one cell in each column (resp. row).

\subsection{$k$-shapes}
\label{SS:kshapes}
The $k$-interior of a partition $\la$ is the subpartition made out
of the cells with hook-length larger than $k$:
$$\Int^k(\la)=\{b\in \la \mid h_\la(b) > k\}\,.$$
The $k$-boundary of $\la$ is the skew shape of cells with hook-length
 bounded by $k$:
$$\bdy^k(\la) =\la/\Int^k(\la)\,.$$
We define the {\it $k$-row shape}
$\rs^k(\la)\in\Z_{\ge0}^\infty$
(resp. {\it $k$-column shape} $\cs^k(\la)\in\Z_{\ge0}^\infty$) of $\la$
to be the sequence giving the number of cells in the rows (resp.
columns) of $\bdy^k(\la)$.
\begin{definition} \label{D:kshape}
Let $k \geq 2$ be an integer.
A partition $\la$ is a {\it $k$-shape} if
$\rs^k(\la)$ and $\cs^k(\la)$ are partitions.  We let
$\Ksh^k$ denote the set of $k$-shapes and $\Ksh^k_N=\{\lambda\in\Ksh^k:
|\bdy^k(\la)|=N\}$.
\end{definition}

\begin{example}\label{X:kshape} The partition
$\la=(8,4,3,2,1,1,1)\in\Ksh^{4}_{12}$,
since $\rs^4(\la)=(4,2,2,1,1,1,1)$ and $\cs^4(\la)=(3,2,2,1,1,1,1,1)$
are partitions and $|\bdy^4(\la)|=4+2+2+1+1+1+1=12$.
The partition $\mu=(3,3,1)\not\in\Ksh^{4}$ since $\rs^4(\mu)=(2,3,1)$ is not a partition.
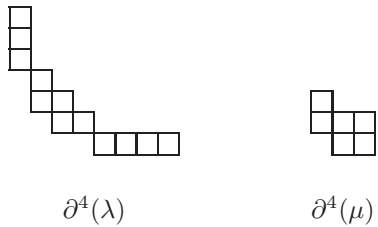
\begin{equation*}
\begin{matrix}
{\scriptsize {\tableau[sby]{\\ \\ \\ \bl& \\ \bl&& \\ \bl&\bl& & \\
\bl&\bl&\bl&\bl&&&&}}} &&&&&
{\scriptsize { \tableau[sby]{\\ && \\ \bl&& }}} \\ \\
\bdy^4(\la) &&&&& \bdy^4(\mu)
\end{matrix}
\end{equation*}
\end{example}
The set of $k$-shapes includes both the  $k$-cores and $k+1$-cores.
\begin{proposition}[\cite{LLMS2}]
$\Core^k\subset \Ksh^k$ and
$\Core^{k+1}\subset\Ksh^k$.
\end{proposition}

\begin{remark}
Since $k$ remains fixed throughout, we shall often for simplicity
suppress $k$
in the notation, writing $\bdy\la$, $\rs(\la)$, $\cs(\la)$, $\Ksh$,
and so forth.
\end{remark}

\subsection{Strings}
The primary notion to define our order on $k$-shapes
is a string of cells lying at a diagonal distance $k$ or $k+1$ from
one another.  To be precise, let $b$ and $b'$ be {\it contiguous}
cells when $|\diag(b)-\diag(b')|\in\{k,k+1\}$.
\begin{remark} \label{R:contig}
Since $\la$-addable cells cannot occur on consecutive diagonals,
a $\la$-addable corner $x$ is contiguous with
at most one $\la$-addable corner above (resp. below) it.
\end{remark}

\begin{definition}\label{D:string}
A {\it string} of {\it length} $\ell$ is a skew shape
$\mu/\la$ which consists of cells $\{a_1,\dots,a_\ell\}$,
where $a_i$ and $a_{i+1}$ are contiguous (with $a_{i+1}$ below $a_i$)
for each $1\leq i<\ell$.
\end{definition}
\begin{example} Let $k=3$, $\lambda=(4,2,1)$
and $\mu=(5,3,1,1)$.  Then $\mu/\lambda$ is a string of length 3.
If we denote each element of $\mu/\lambda$ by a $\bullet$, the
string can be represented as:
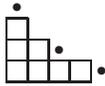
\begin{equation*}
{\scriptsize {\tableau[sby]{ \bl \bullet \\   \\ & & \bl \bullet
 \\  &
&& & \bl \bullet  \\}}}
\end{equation*}
\end{example}

Given a string $s=\mu/\lambda=\{a_1,\dotsc,a_\ell\}$,
of particular importance are certain columns and rows
called modified rows.  Define
$\change_\rs(s)=\rs(\mu)-\rs(\la)\in\Z^\infty$.
The {\it positively} (resp. {\it negatively})
{\it modified rows} of $s$ are those corresponding to
positive (resp. negative) entries in $\change_\rs(s)$.
We define in a similar way $\change_\cs(s)$, and
the {positively} (resp. {negatively})
{modified columns} of $s$.
Any string $s=\mu/\la$
can be categorized into one of four types
\begin{itemize}
\item A {\it row-type string} if $s$ does not have positively or
negatively modified rows
\item A {\it column-type string} if $s$ does not have positively or
negatively modified columns.
\item A {\it cover-type string} if $s$ has a positively modified row and
a positively modified column
\item A {\it cocover-type string} if $s$ has a negatively modified row and
a negatively modified column
\end{itemize}
It is helpful to depict a string $s=\mu/\la$ by its
{\it diagram}, defined by the following data: cells
of $s$ are represented by the symbol $\bullet$,
cells of $\bdy\la\setminus\bdy\mu$ are represented by $\circ$,
and cells of $\bdy\mu\cap\bdy\la$ in
the same row (resp. column) as some $\bullet$ or $\circ$
are collectively depicted by
a horizontal (resp. vertical) line segment. The four possible
string diagrams are shown in Figure \ref{F:stringdiags}.
\begin{figure}
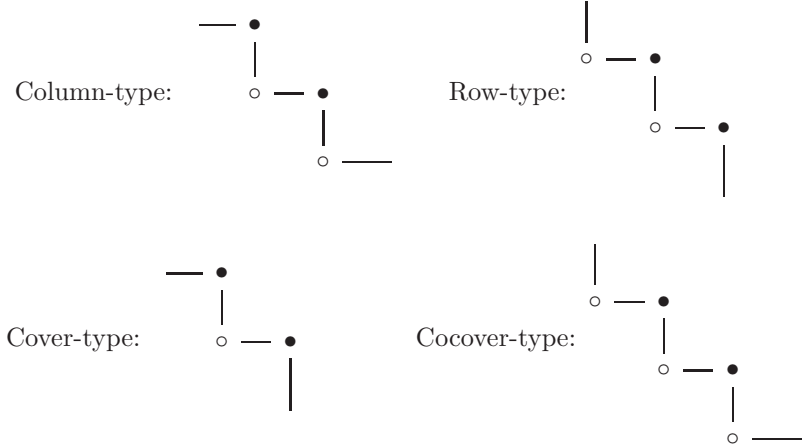

\dgARROWLENGTH=1.1em
$$
\text{Column-type:}
\begin{diagram}
\node{} \arrow{e,-} \node{\bullet} \arrow{s,-}\\ \node[2]{\circ}
\arrow{e,-} \node{\bullet} \arrow{s,-}\\ \node[3]{\circ} \arrow{e,-}
\end{diagram}
\qquad \qquad \text{Row-type:}
\begin{diagram}
\node{} \arrow{s,-}\\ \node{\circ} \arrow{e,-} \node{\bullet}
\arrow{s,-}\\ \node[2]{\circ} \arrow{e,-} \node{\bullet}
\arrow{s,-}\\ \node{}
\end{diagram}
$$
$$
\text{Cover-type:}
\begin{diagram}
\node{} \arrow{e,-} \node{\bullet} \arrow{s,-}\\ \node[2]{\circ}
\arrow{e,-} \node{\bullet} \arrow{s,-}\\ \node{}
\end{diagram}
\qquad \qquad \text{Cocover-type:}
\begin{diagram}
\node{} \arrow{s,-}\\ \node{\circ} \arrow{e,-} \node{\bullet}
\arrow{s,-}\\ \node[2]{\circ} \arrow{e,-} \node{\bullet} \arrow{s,-}
\\ \node[3]{\circ} \arrow{e,-}
\end{diagram}
$$
\caption{Types of string diagrams}
\label{F:stringdiags}
\end{figure}

\subsection{Moves}
The covering relations in the poset of $k$-shapes will be defined by letting a $k$-shape $\lambda$ be larger
than the $k$-shape $\mu$ when the skew diagram $\mu/\lambda$ is a
particular succession of strings (called a move).  To this end,
define two strings to be
{\it translates}
when they are translates of each other in $\Z^2$ by a fixed vector,
and their corresponding modified rows and columns agree in size.
Equivalently, two strings are translate if
their diagrams have the property that
$\bullet$'s and $\circ$'s appear in the same relative
positions with respect to each other and the
lengths of each corresponding horizontal and vertical segment
are the same.  We will also refer to cells $a_j$ and
$b_j$ as translates when strings $s_1=\{a_1,\ldots,a_\ell\}$
and $s_2=\{b_1,\ldots,b_\ell\}$ are translates.


\begin{definition} \label{D:defmove}
A {\it row move} $m$ of rank $r$ and length $\ell$
is a chain of partitions
$\la=\la^0\subset\la^1\subset\dotsm\subset\la^r=\mu$
that meets the following conditions:
\begin{enumerate}
\item \label{I:movestart} $\la\in\Ksh$
\item \label{I:movestrings}
$s_i=\la^i/\la^{i-1}$ is a row-type string consisting of $\ell$ cells
for all $1\le i\le r$
\item \label{I:movetranslates} the strings $s_i$ are translates of each other
\item \label{I:movetopcells} the top cells of $s_1, \dotsc, s_r$ occur
in consecutive columns
from left to right
\item \label{I:moveend} $\mu\in\Ksh$.
\end{enumerate}
We say that $m$ is a row move from $\la$ to $\mu$ and
write $\mu=m*\la$ or $m=\mu/\la$.
A {\it column move} is the transpose analogue of a row move.
A {\it move} is a row move or column move.  It should be noted that it
can be shown that moves are at most
of rank $k-1$ \cite{LLMS2}.
\end{definition}

\begin{example} \label{X:rowmove}
For $k=5$, a row move of length $1$
and rank $3$ with strings $s_1=\{A\}$, $s_2=\{B\}$,
and $s_3=\{C\}$ is pictured below. The lower case
letters are the cells that are removed
from the $k$-boundary when the
corresponding strings are added.
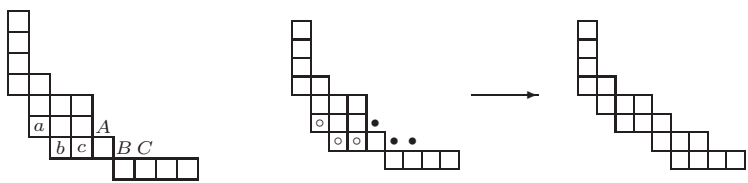
\begin{equation*}
\begin{diagram}
\node{{\scriptsize \tableau[sby]{\\ \\ \\ & \\ \bl&&& \\ \bl&a&&&\bl A
\\\bl&\bl&b&c&&\bl B& \bl C \\ \bl&\bl&\bl&\bl&\bl&&&&}} }
\node{\tiny\tableau[sby]{\\ \\ \\ & \\ \bl&&&\\ \bl&\circ&& & \bl \bullet\\
\bl&\bl&\circ&\circ& & \bl \bullet & \bl \bullet\\ \bl&\bl&\bl&\bl&\bl&&&&}}
\arrow{e,->}
\node{\quad \tiny\tableau[sby]{\\ \\ \\ & \\ \bl&&&\\ \bl&\bl&&& \\
\bl&\bl&\bl&\bl&&& \\ \bl&\bl&\bl&\bl&\bl&&&&}}
\end{diagram}
\end{equation*}
For $k=3$, a row move of length $2$ and rank $2$
with strings $s_1=\{A_1,A_2\}$ and $s_2=\{B_1,B_2\}$ is:
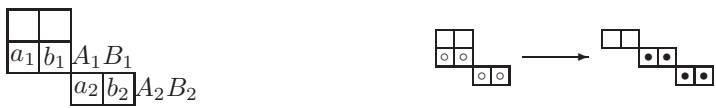
\begin{equation*}
\begin{diagram}
\node{{ \tableau[sby]{&\\ a_1&b_1&\bl A_1&\bl B_1 \\
\bl&\bl&a_2&b_2&\bl A_2 & \bl B_2}}}
\node[2]{\tiny\tableau[sby]{&\\
\circ&\circ \\ \bl&\bl&\circ&\circ}} \arrow{e,->} \node{\tiny\tableau[sby]{& \\
\bl&\bl&\bullet&\bullet \\ \bl&\bl&\bl&\bl&\bullet&\bullet }}
\end{diagram}
\end{equation*}
\end{example}
The rationale for the ``move'' terminology is the following.
Suppose that $m$ is a row move from $\lambda$ to $\mu$.
By definition, $\rs(\mu)=\rs(\la)$  (since the strings are all of row-type).
Hence $\bdy\mu$ can be viewed as a right-shift (or move) of certain rows of
$\bdy \la$.

\section{The poset of $k$-shapes} \label{SubS:poset}
We now define a poset structure, called the poset of $k$-shapes, on the set $\Ksh_N$ of $k$-shapes of fixed size $N$.
We say that $\lambda$ dominates $\mu$ in the poset of $k$-shapes
 if
there is a sequence of moves  $m_1,\dots,m_r$ such that
$\mu = m_r * \cdots * m_1 * \lambda$.  We let $\mathcal P^k(\lambda,\mu)$
denote the set of paths in the poset of $k$-shapes from $\lambda$ to $\mu$.
\begin{proposition}[\cite{LLMS2}]  An element of the \kshapeposet~ is maximal
(resp. minimal)
if and only if it is a
$(k+1)$-core (resp. $k$-core).
\end{proposition}

\begin{example} \label{X:kshapeposet} The poset of $2$-shapes of size 4 is pictured below. Only the
cells of the $k$-boundaries are shown. Row moves are indicated by $r$ and column moves by $c$.
\dgARROWLENGTH=2.5em
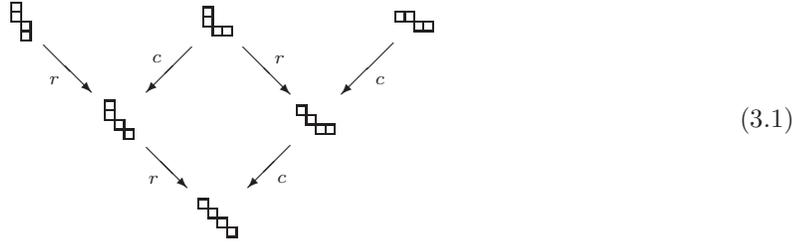
\begin{equation}
\begin{diagram} \label{E:poset2,4}
\node{\tableau[pby]{ \\ \\ \bl& \\ \bl& }} \arrow{se,b}{r} %
\node[2]{\tableau[pby]{ \\ \\ \bl& & }} \arrow{sw,t}{c} \arrow{se,t}{r} %
\node[2]{\tableau[pby]{ & \\ \bl&\bl&&}} \arrow{sw,b}{c} \\
\node[2]{\tableau[pby]{ \\ \\ \bl & \\ \bl&\bl& }} \arrow{se,b}{r} %
\node[2]{\tableau[pby]{\\ \bl&\\ \bl&\bl&&}} \arrow{sw,b}{c} \\
\node[3]{\tableau[pby]{ \\ \bl& \\ \bl&\bl& \\ \bl&\bl&\bl&}}
\end{diagram}
\end{equation}
The poset of $3$-shapes of size 5
is pictured below.
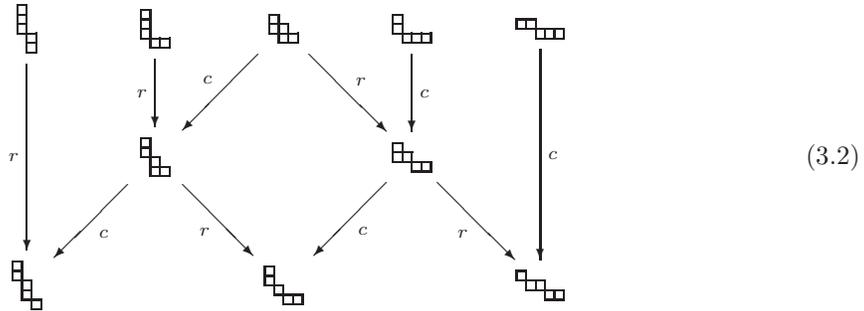
\begin{equation}
\begin{diagram} \label{E:poset3,5}
\node{\tableau[pby]{ \\ \\ \\ \bl& \\ \bl& }} \arrow[2]{s,t}{r} %
\node{\tableau[pby]{ \\ \\ \\ \bl&&}} \arrow{s,t}{r}  %
\node{\tableau[pby]{ \\ & \\ \bl&& }} \arrow{sw,t}{c} \arrow{se,t}{r} %
\node{\tableau[pby]{ \\ \\ \bl &&&}} \arrow{s,b}{c} %
\node{\tableau[pby]{ & \\ \bl&\bl&&& }} \arrow[2]{s,b}{c} \\
\node[2]{\tableau[pby]{\\ \\ \bl& \\ \bl & & }} \arrow{sw,b}{c} \arrow{se,b}{r} %
\node[2]{\tableau[pby]{\\ & \\ \bl& \bl & & }} \arrow{sw,b}{c} \arrow{se,b}{r} \\
\node{\tableau[pby]{ \\ \\ \bl& \\ \bl&\\ \bl&\bl&}} %
\node[2]{\tableau[pby]{\\ \\ \bl& \\ \bl&\bl&&}} %
\node[2]{\tableau[pby]{\\ \bl&& \\ \bl&\bl&\bl&&}}
\end{diagram}
\end{equation}
\end{example}

\begin{definition} \label{D:charge}
Given a move $m$, the {charge}  of $m$, written $\charge(m)$, is
$0$ if $m$ is a row move and $r\ell$ if $m$ is
a column
move of length $\ell$ and rank $r$. Notice that in the column case,
$r \ell$ is simply the number of cells in the move $m$
when viewed as a skew shape.  The charge of a path $\bp =
(m_1,\dots,m_n)$ in $\Ksh_N$
is
$\charge(m_1)+\cdots+\charge(m_n)$, the sum of the charges of the moves that
constitute the path.
\end{definition}
\begin{definition}Given a move $m$, the {cocharge}  of $m$, written $\cocharge(m)$, is
$0$ if $m$ is a column move and $r\ell$ if $m$ is
a row
move of length $\ell$ and rank $r$.  The cocharge of a path is again the sum of
the cocharges of the moves forming the path.
\end{definition}

Let $\equiv$ be the equivalence relation on paths in $\Ksh_N$
generated by the following {\it diamond equivalences}:
\begin{equation} \label{E:elemequiv}
\tM m \equiv \tm M
\end{equation}
where $m,M,\tm,\tM$ are moves (possibly empty) between $k$-shapes
such that the diagram
 \dgARROWLENGTH=2.5em
\begin{equation} \label{E:commudiagram}
\begin{diagram}
\node[2]{\la} \arrow{sw,t}{m} \arrow{se,t}{M} \\ \node{\mu}
\arrow{se,b}{\tilde M} \node[2]{\nu} \arrow{sw,b}{\tilde m} \\
\node[2]{\gamma}
\end{diagram}
\end{equation}
commutes and the charge is the same on both sides of the diamond:
\begin{equation} \label{E:charge}
\charge(m)+\charge(\tM)=\charge(M)+\charge(\tm).
\end{equation}
The commutation is equivalent to the equality $\tM\cup m = \tm\cup
M$ where a move is regarded as a set of cells.  Observe that
the charge is by definition constant
on equivalence classes of paths.  We will let $\overline{\mathcal P}^k(\lambda,\mu)$ be the set of equivalences classes in ${\mathcal P}^k(\lambda,\mu)$,
that is, the set of equivalences classes of paths in the poset of $k$-shapes from
$\lambda$ to $\mu$.  It is easy to see that
\begin{equation}  \label{E:cocharge}
\cocharge(m)+\cocharge(\tM)=\cocharge(M)+\cocharge(\tm) \iff
\charge(m)+\charge(\tM)=\charge(M)+\charge(\tm)
\end{equation}
and thus  \eqref{E:charge} can be replaced by  the left hand side of
\eqref{E:cocharge} in the definition of diamond equivalence.
\begin{example}
Continuing Example \ref{X:kshapeposet}, the two paths in the poset
of 2-shapes from $\la=(3,1,1)$ to
$\mu=(4,3,2,1)$ have charge 2 and 3 respectively, and so are not equivalent.
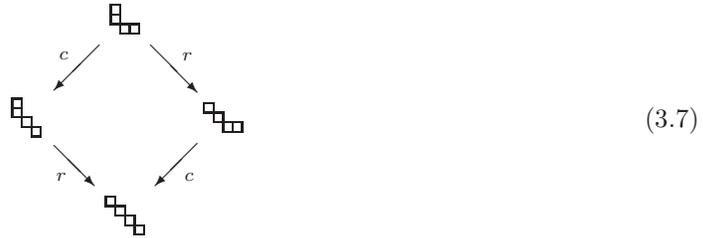
\dgARROWLENGTH=2.5em
\begin{equation}
\begin{diagram}
\node[2]{\tableau[pby]{ \\ \\ \bl& & }} \arrow{sw,t}{c} \arrow{se,t}{r} \\
\node{\tableau[pby]{ \\ \\ \bl & \\ \bl&\bl& }} \arrow{se,b}{r} %
\node[2]{\tableau[pby]{\\ \bl&\\ \bl&\bl&&}} \arrow{sw,b}{c} \\
\node[2]{\tableau[pby]{ \\ \bl& \\ \bl&\bl& \\ \bl&\bl&\bl&}}
\end{diagram}
\end{equation}

The two paths in the poset of 3-shapes from $\la=(3,2,1)$ to
$\nu=(4,2,1,1)$ are diamond equivalent, both having charge 1.

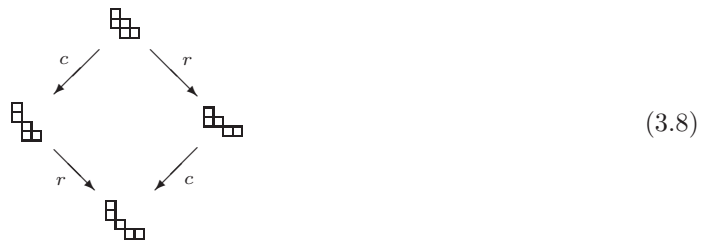
\begin{equation}
\begin{diagram} 
\node[2]{\tableau[pby]{ \\ & \\ \bl&& }} \arrow{sw,t}{c} \arrow{se,t}{r} \\
\node{\tableau[pby]{\\ \\ \bl& \\ \bl & & }}\arrow{se,b}{r} %
\node[2]{\tableau[pby]{\\ & \\ \bl& \bl & & }} \arrow{sw,b}{c}  \\
\node[2]{\tableau[pby]{\\ \\ \bl& \\ \bl&\bl&&}} %
\end{diagram}
\end{equation}
\end{example}

\section{(Co)charge of a $k$-tableau} \label{Schargektab}
A $k$-tableau (or weak tableau) is a special type of tableau
originally introduced to describe the Pieri-type rules that the $k$-Schur functions satisfy.  The dual $k$-Schur functions are the generating
series of $k$-tableaux of a given shape.
\begin{definition} A $k$-tableau of weight $(\alpha_1,\dots,\alpha_N)$
is a sequence of $k+1$-cores
$\emptyset=\lambda^{(0)} \subseteq \lambda^{(1)} \subseteq \cdots \subseteq \lambda^{(N)}=\lambda$ such that, for all $i$,
$\rs(\lambda^{(i)})/\rs(\lambda^{(i-1)})$ is a horizontal strip
and $\cs(\lambda^{(i)})/\cs(\lambda^{(i-1)})$ is a vertical strip, both of size
$\alpha_i$.
\end{definition}

\begin{example} Let $\alpha = (2,3,1,2,2)$.  An example of a 3-tableau of weight $\alpha$ is
$$
 {\small  \tableau[scY]{4,5|3,4|2,2,4,5,5|1,1,2,2,2,4,5,5}}
$$
The corresponding sequence of 4-cores (represented by $\partial (\lambda^{(i)})$
to view more easily $\rs(\lambda^{(i)})$
and $\cs(\lambda^{(i)})$) is

\begin{tabular}{c@{\hspace{.5in}}c@{\hspace{.5in}}c@{\hspace{.5in}}c@{\hspace{.5in}}c}
 {\tiny $ \tableau[scY]{,}$} &  {\tiny $ \tableau[scY]{,|\bl,\bl,,,} $ }&
{\tiny $\tableau[scY]{|,|\bl,\bl,,,}$}
& {\tiny $\tableau[scY]{|,|\bl,,|\bl,\bl,\bl,,,}$} &  {\tiny $ \tableau[scY]{,|,|\bl,\bl,,,|\bl,\bl,\bl,\bl,\bl,,,}$} \\ \\
 $\lambda^{(1)}$ & $\lambda^{(2)}$ & $\lambda^{(3)}$ & $\lambda^{(4)}$ & $\lambda^{(5)}$ \\
 \end{tabular}

\noindent It can be checked that
$\rs(\lambda^{(i)})/\rs(\lambda^{(i-1)})$ is a horizontal strip
and $\cs(\lambda^{(i)})/\cs(\lambda^{(i-1)})$ is a vertical strip both of size $\alpha_i$ for $i=1,2,3,4,5$.
\end{example}

The $k+1$-residue (or simply residue if the value of $k$
is obvious from the context) of a cell
$b=(i,j)$ is equal to $j-i \mod k+1$. It can be shown \cite{LM:cores}
that a $k$-tableau of weight $(\alpha_1,\dots,\alpha_N)$ is such that
the cells occupied by letter $i$ in the $k$-tableau
have exactly $\alpha_i$ distinct residues.

For this article, we will mostly
need $k$-tableaux of weight $(1,1,\dots,1)$, that is, standard $k$-tableaux.
To be more specific:
\begin{definition} A standard $k$-tableau is a sequence of $k+1$-cores
$\emptyset=\lambda^{(0)} \subseteq \lambda^{(1)} \subseteq \cdots \subseteq \lambda^{(N)}=\lambda$ such that, for all $i$,
$\lambda^{(i)}/\lambda^{(i-1)}$ is a vertical and a horizontal strip and
$|\partial(\lambda^{(i)})|=|\partial(\lambda^{(i-1)})|+1$.
\end{definition}
Our previous observation then says that
the cells occupied by
letter $i$ in a standard $k$-tableau all have the same residue.

\begin{example}\label{ejemplo1}
Let $k=3$.  An example of a $3$-tableau is $T =$
{\tiny \tableau[scY]{6|4|3|2,6,7|1,4,5}}

The corresponding sequence of $4$-cores (with the residue of each letter)
is

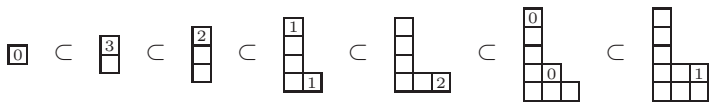
\begin{tabular}{ccccccccccccc} {\tiny{\tableau[scY]{ 0}}} & $\subset$
& {\tiny \tableau[scY]{ 3 | $$ }} & $\subset$ &
{\tiny{\tableau[scY]{2| | $$}}} & $\subset$ &
{\tiny{\tableau[scY]{1| | | ,1}}} & $\subset$ &
{\tiny{\tableau[scY]{| | | , ,2}}} & $\subset$ &
{\tiny{\tableau[scY]{0|||,0|,,}}} & $\subset$ &
{\tiny{\tableau[scY]{|||,,1|,,}}}
\end{tabular}
\end{example}

We now introduce concepts of charge and cocharge for
$k$-tableaux.   The cell corresponding to the lowermost (resp. uppermost)
occurrence of a given letter $n$ will be denoted
 $\ndown$ (resp.  $\nup$).   For instance, the cell
$6^\downarrow$ is the marked one in the tableau
{\tiny{\tableau[scY]{ 6|4|3|2,\tf 6, 7|1,4,5}}}.  We will use
$n_-^{\uparrow}$ and $n_+^{\uparrow}$ to denote respectively $(n-1)^\uparrow$ and
$(n+1)^\uparrow$ (and similarly for $n_-^{\downarrow}$ and $n_+^{\downarrow}$).
Given
two cells $b_1$ and $b_2$ of a $(k+1)$-core $\lambda$
such
that $b_2$ is weakly below $b_1$, we let
${\rm diag}_e(b_1,b_2)$  be the number of diagonals of residue $e$
strictly between $b_1$ and $b_2$.  The charge of a standard
$k$-tableau $T$ on $N$
letters is
\begin{equation}
{\charge} (T) = \sum_{n=1}^N \charge(n)
\end{equation}
where $\charge(1)=0$, and where $\charge(n)$ for $n>1$ is defined recursively
in the following way.   Suppose that
$n^\uparrow$ and $n_-^\uparrow$ have residues
$e$ and $e_-$ respectively.  Then $\charge(n)$ is defined as
\begin{equation}\label{defcharge1}
 \ch(n) = \left \{ \begin{array}{lll}
                    \ch(n-1) + {\rm diag}_{e_-}(n_-^\uparrow,n^\uparrow) +1 & \text{if } & n_-^\uparrow \text{ is weakly above }  n^\uparrow \\ \\
                    \ch(n-1) - {\rm diag}_e(n^\uparrow ,n_-^\uparrow) & \text{if } &n_-^\uparrow \text{ is below } n^\uparrow \\
                    \end{array} \right.
\end{equation}

\begin{example}\label{exa21}
Consider the 4-tableau with the cells $\nup$ marked
 \quad $T = {\small{\tableau[scY]{\tf 10|\tf 8|\tf 5 ,\tf 7|\tf 4,\tf 6,10|\tf 1,\tf 2,\tf 3,5,7,\tf9,10 }}}$

\noindent Adding the residues to the core

  $$ {\tiny{\tableau[scY]{ 1  |  2|  3,  4| 4,  0, 1 |  0,  1,  2, 3, 4,  0,1 }}}$$

\noindent we obtain

\hspace{-0.2in}
$\begin{array}{llll}
  \ch(1) = 0 \; ; & \quad \ch(2) = \ch(1)+1=1 \; ;&   \ch(3)=\ch(2)+1=2 \; ; &   \quad \ch(4)=\ch(3)=2 \; \\
  \ch(5)= \ch(4) = 2 \; ;& \quad \ch(6)= \ch(5) +1 = 3 \; ;&  \ch(7)= \ch(6)= 3 \; ;& \quad \ch(8)= \ch(7) = 3 \; ; \\
  \multicolumn{2}{l}{\ch(9)= \ch(8) + {\rm diag}_2(8^\uparrow ,9^\uparrow)+1 =  3+1+1=5 \; ; } & \multicolumn{2}{l}{ \ch(10)= \ch(9) - {\rm diag}_1(10^\uparrow ,9^\uparrow)+1 =  5-1=4 }\\
 \end{array} $

\noindent Hence $ \ch(T) = 0+1+2+2+2+3+3+3+5+4= 25$.
\end{example}

Similarly, the cocharge of a standard
$k$-tableau $T$ on $N$
letters is
\begin{equation}
{\cocharge} (T) = \sum_{n=1}^N \cocharge(n)
\end{equation}
where $\cocharge(1)=0$, and where $\cocharge(n)$ for $n>1$ is defined recursively as (supposing that
$n^\downarrow$ and $n_-^\downarrow$ have residues
$e$ and $e_-$ respectively)
\begin{equation}\label{defcocharge1}
 \cocharge(n) = \left \{ \begin{array}{lll}
                    \cocharge(n-1) - {\rm diag}_{e_-}(n_-^\downarrow,n^\downarrow) & \text{if } & n_-^\downarrow \text{ is weakly above }  n^\downarrow \\ \\
                    \cocharge(n-1) + {\rm diag}_e(n^\downarrow ,n_-^\downarrow) +1& \text{if } &n_-^\downarrow \text{ is below } n^\downarrow \\
                    \end{array} \right.
\end{equation}

\begin{example} Consider the 4-tableau in Example~\ref{exa21} with this time the cells $\ndown$ marked

$$T = {\small{\tableau[scY]{ 10 | \tf 8 | 5, 7 | \tf 4 , \tf 6 , 10| \tf 1 , \tf 2 , \tf 3 , \tf 5 ,  \tf 7  ,  \tf 9 , \tf 10 }}}$$

\noindent We have $\cocharge(1)=0; \quad
\cocharge(2)= \cocharge(1)-{\rm diag}_4(1^\downarrow ,2^\downarrow) = 0; \quad
\cocharge(3)= \cocharge(2)-{\rm diag}_0(2^\downarrow ,3^\downarrow) = 0; \quad
\cocharge(4)= \cocharge(3)+{\rm diag}_4(4^\downarrow ,3^\downarrow) = 1; \quad
\cocharge(5)= \cocharge(4)-{\rm diag}_3(4^\downarrow ,5^\downarrow) = 1; \quad
\cocharge(6)= \cocharge(5)+{\rm diag}_0(6^\downarrow ,5^\downarrow) = 2; \quad
\cocharge(7)= \cocharge(6)-{\rm diag}_4(6^\downarrow ,7^\downarrow) = 2; \quad
\cocharge(8)= \cocharge(7)+{\rm diag}_1(8^\downarrow ,7^\downarrow) = 4; \quad
\cocharge(9)= \cocharge(8)-{\rm diag}_1(8^\downarrow ,9^\downarrow) = 3; \quad
\cocharge(10)= \cocharge(9)-{\rm diag}_4(9^\downarrow ,10^\downarrow) = 3$.
 Hence $\cocharge(T)=0+0+0+1+1+2+2+4+3+3= 16$.

\end{example}

\begin{remark}  That the charge and cocharge are given
by nonnegative integers follows from their compatibility with the
weak bijection.  In effect, iterating the weak bijection (as was done for instance
in the introduction) puts in correspondence a $k$-tableau
$Q^{(k)}$ with a sequence of paths $([\bp_k],[\bp_{k-1}],\dots,[\bp_2])$
such that $\charge(Q^{(k)})=\charge(\bp_k)+\charge(\bp_2)$ and
 $\cocharge(Q^{(k)})=\cocharge(\bp_k)+\cocharge(\bp_2)$, from which the
nonnegativity is immediate.
 \end{remark}

The remainder of this section is concerned with the definition of charge
in the non-standard case.   Our initial goal was to show that charge was
also compatible with the weak bijection in the non-standard case,
but as is discussed in the conclusion,
we were unfortunately not able to reach that goal.

 The definition of charge is given
for $k$-tableaux of dominant weights, and as in the usual case,
a Lascoux-Sch\"utzenberger type action of the symmetric group on $k$-tableaux
allows to send a $k$-tableau of any weight into a $k$-tableau of
dominant weight.  We do not prove here that the definition of charge
and of the Lascoux-Sch\"utzenberger type action of the symmetric group on $k$-tableaux are well-defined.

We first define the charge of a $k$-tableau $T$ in the non-standard case.
First suppose
that the weight $(\alpha_1,\dots,\alpha_N)$ of the $k$-tableau is dominant,
that is, that $\alpha_1 \geq \alpha_2 \geq \cdots \geq \alpha_N$.  As mentioned earlier, the letter $n$ in $T$ occupies exactly $\alpha_n$ residues
$i_1,\dots,i_{\alpha_n}$.  We can thus denote the letters $n$ in $T$ as
$n_{i_1},\dots, n_{i_{\alpha_n}}$.  We now form $\alpha_1$ words in the following way.
Start with $1_{j_1}=1_{\alpha_1-1}$ (the rightmost 1 in $T$) and construct a word recursively by appending to
$1_{j_1} 2_{j_2} \cdots n_{j_n}$ the letter $(n+1)_{j_{n+1}}$, where $j_{n+1}$
is the largest element in the total order $j_n+1<j_n+2<\cdots<j_n-1$ (the residues are taken modulo $k+1$). Note that it can be shown that there
is such a letter.
Once $w_1=1_{j_1} 2_{j_2} \cdots N_{j_N}$ has been constructed, remove all the letters $1_{j_1},2_{j_2}, \dots, N_{j_N}$ from $T$, and
construct $w_2$ in the same manner starting this time with $1_{\alpha_1-2}$ (the second rightmost 1 in $T$) and stopping at the largest letter.
After all the words $w_1,\dots,w_{\alpha_1}$ have
been constructed, the charge of $T$ is the sum of the charges of the $w_i$'s,
where the charge of $w_i$ is the charge of the subtableau of $T$
 obtained by considering only the letters in $w_i$.  For instance,
let $k=4$ and consider the $4$-tableau of evaluation $(2,2,2,2,2,2,1)$
$$
T= {\small  \tableau[scY]{7|6|5,6|3,4,7|2,3,5,5,6|1,1,2,3,4,4,5,5,6}}
$$
The letters are $1_0, 1_1, 2_2,2_4, 3_0,3_3, 4_0, 4_4,5_1,5_2,6_1,6_3,7_0$,
from which we extract $w_1=1_1 2_4 3_3 4_0 5_2 6_1 7_0$ and $w_2=1_0 2_2 3_0 4_4 5_1 6_3$.  The charge of $w_1$ is computed using the cells
$$
 {\small  \tableau[scY]{\tf 7|\tf 6|\tf 5,6|\tf 3,4,7|\tf 2,3,5,5,6|1,\tf 1,2,3,4,\tf 4,5,5,6}}
$$
Hence $\charge(w_1)=0+0+0+2+1+1+1=5$.  Similarly, the charge of $w_2$ is computed using the cells
$$
 {\small  \tableau[scY]{7|6|5,\tf 6|3,\tf 4,7|2,\tf 3,\tf 5,5,6|\tf 1,1,\tf 2,3,4,4,5,5,6}}
$$
and is such that $\charge(w_2)=0+1+1+1+2+2=7$.  The charge of $T$ is therefore equal to $5+7=12$.

If the weight of $T$ is not dominant, we define the charge of $T$ to be the
charge of $\sigma(T)$, where $\sigma(T)$ is the unique $k$-tableau of dominant
weight obtained
by the following (conjectural) Lascoux-Sch\"utzenberger type action of the symmetric group on $k$-tableaux.   Consider the elementary transposition $\sigma_i$
which sends a tableau of weight $(\alpha_1,\dots,\alpha_N)$ to a tableau of weight $(\alpha_1,\dots,\alpha_{i-1},\alpha_{i+1},\alpha_i,\alpha_{i+2},\dots\alpha_N)$.
Let $i$ and $i+1$ be denoted respectively as $a$ and $b$.
Suppose that the $a$'s and $b$'s in
$T$ occupy residues ${i_1},\dots, {i_r}$ and $j_1,\dots,j_s$ respectively.
We say that $b_j$ lies on the floor if there is a $b_j$ in $T$ without an
$a$ below it (that is, one unit downward).  If $b_j$ does not lie on the floor then let $b_j=\tilde b_{j+1}$,
otherwise let $b_j=\tilde b_j$.  Order the $a$'s and $\tilde b$'s using the total order
$$
\tilde b_0 < a_0 < \tilde b_1 < a_1 <\cdots < \tilde b_k < a_k
$$
and let the word corresponding to the the ordered
$a$'s and $\tilde b$'s be $w$.
Then do the usual
pairing of the
letters $a$'s
 and $\tilde b$'s followed by a
permutation of their  weight to obtain a word $w'$.
That is,
pair every factor $\tilde b a$ of $w$, and let $w_1$ be the subword of $w$
made out of the unpaired letters.  Pair every factor $\tilde b\,a$ of $w_1$, and
let $w_2$ be the subword made out of the unpaired letters.  Continue
in this fashion as long as possible.
When all factors $\tilde b\,a$ are paired and unpaired letters of $w$
are of the form $a^r \tilde b^s$, send $a^r \tilde b^s$ to $a^s \tilde b^r$
(keeping track of the residues). Now the $a$'s in $w'$ will have residue
$i_1',\dots,i_s'$.  Take $T_{i-1}'=T_{i-1}$ and let $T'_i$ be the unique tableau such that letter $i$ occupies residues $i_1',\dots,i_s'$.
Let $T'_{i+1}$ be the unique tableau of the same shape as
$T_{i+1}$ with subtableau $T'_i$.  Finally, let $T'$ be obtained by
filling the rest of the tableau as in $T$.  We then define
$\sigma_i (T)= T'$.  Since the $\sigma_i$'s generate the symmetric group,
this defines an action of the symmetric group on $k$-tableaux.  It is proven
in \cite{MEtesis} (in a more general context) that $\sigma_i$ is an
involution that sends a $k$-tableau into a $k$-tableau.  The fact that the $\sigma_i$'s obey the Coxeter relations was always in our mind a consequence of the yet
unproven compatibility of the Lascoux-Sch\"utzenberger type action of the symmetric group on $k$-tableaux with the weak bijection (in which case the Coxeter relations would follow from the known Coxeter relations when $k$ is large), and
as such we never intended to prove them.

\section{Standard $k$-shape tableaux} \label{Skshapetab}

An generalization of $k$-tableaux (fillings of $k+1$-cores)
 to certain fillings of $k$-shapes called $k$-shape tableaux
was introduced in \cite{LLMS2}.  For the purposes of this article,
we will only need to describe explicitly the standard case.

\begin{definition}  We say that a string $s=\mu/\lambda=\{a_1,\dots,a_\ell \}$
can be continued below (resp. above) if there is an addable corner
of $\lambda$ below (resp. above) the string $s$ that is contiguous to  $a_{\ell}$ (resp. $a_1$).
We say that a cover-type string is maximal if it cannot be continued above
or below.
\end{definition}

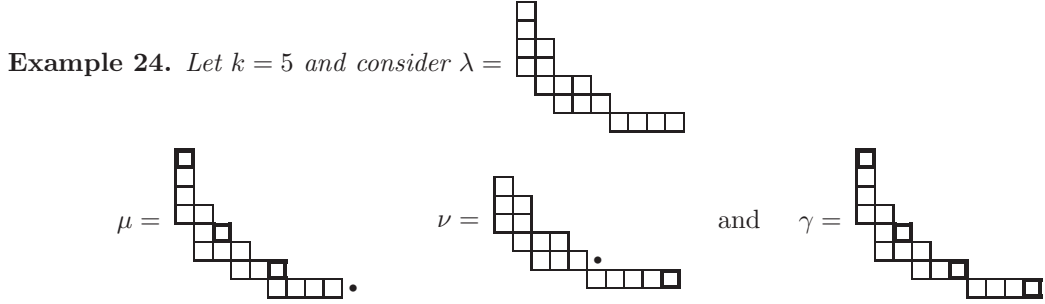
\begin{example}
Let $k=5$ and consider
$\lambda = \; {\tiny {\tableau[scy]{ \\ \\ ,\\ , \\ \bl& , ,\\ \bl& \bl& , , \\ \bl& \bl& \bl& \bl& \bl& , , , }}} $

$$\mu = \; {\tiny {\tableau[scy]{ \tf \\ \\ \\ ,\\ \bl& , \tf\\ \bl& , ,\\ \bl& \bl& \bl& , , \tf \\ \bl& \bl& \bl& \bl& \bl& , , , , \bl {\, \bullet} } \qquad \qquad \nu = \;  \tableau[scy]{ \bl \\ \\ ,\\ , , \bl {\, } \\ \bl& , ,\\ \bl& \bl& , , , \bl {\, \bullet} \\ \bl& \bl& \bl& \bl& \bl& , , , , \tf}  \qquad {\rm and} \qquad \gamma = \;  \tableau[scy]{\tf \\ \\ \\ ,\\ \bl& , \tf \\ \bl& , ,\\ \bl& \bl& \bl& , , \tf \\ \bl& \bl& \bl& \bl& \bl& \bl& , , ,  \tf }}}
$$
\noindent In this case $\mu /\lambda$ (indicated by framed boxes in the diagram)
is a string that can be continued below,  $\nu / \lambda$ is a string that can be continued above and $\gamma / \lambda$ is a maximal cover-type string. We denoted by $\bullet$ the contiguous addable corner below or above.
\end{example}

\begin{definition}  We say that a string $s=\mu/\lambda=\{a_1,\dots,a_\ell \}$
can be {\it reverse-}continued below (resp. above) if there is a removable corner
of $\mu$ below (resp. above) the string $s$ that is contiguous to  $a_{\ell}$ (resp. $a_1$).
We say that a cover-type string is reverse-maximal if it cannot be reverse-continued above or below.
\end{definition}

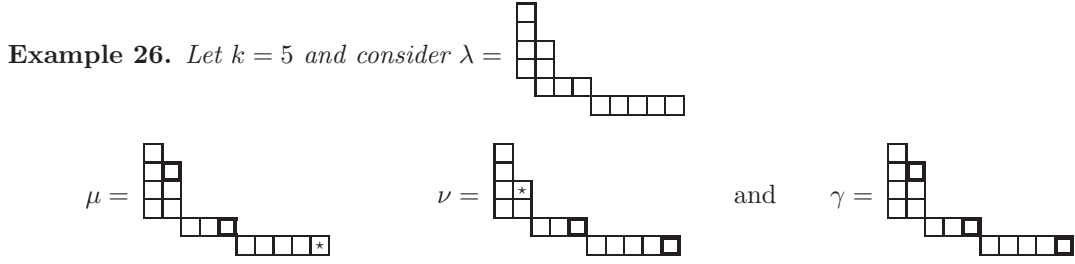
\begin{example}
Let $k=5$  and consider
$\lambda = \;  {\tiny {\tableau[scy]{ \\ \\ ,\\ , \\ \bl& , ,\\ \bl& \bl& \bl& \bl& , , , , }}}  $

$$\mu = \; {\tiny { \tableau[scy]{ \\ , \, \tf \\ ,\\ , \\ \bl& \bl& , , \, \tf \\ \bl& \bl& \bl& \bl& \bl& , , , , \star}}}  \qquad \qquad \nu = \;  {\tiny {\tableau[scy]{ \\ \\ , \star\\ , \\ \bl& \bl& , , \, \tf \\ \bl& \bl& \bl& \bl& \bl& , , , , \, \tf }}} \qquad {\rm and} \qquad \gamma = \;  {\tiny {\tableau[scy]{ \\ , \, \tf \\ ,\\ , \\ \bl& \bl& , , \, \tf \\ \bl& \bl& \bl& \bl& \bl& , , , , \, \tf } }}$$
\noindent
We have that $\mu / \lambda$ is a string that can be reverse-continued below,
$\nu / \lambda$ is a string that can be reverse-continued above and $\gamma/\lambda$ is a reverse-maximal cover-type string. We denoted by $\star$ the contiguous removable corners below or above.
  \end{example}

\begin{definition} We say that $\mu/\lambda$ is a cover if $\lambda$ and
$\mu$ are $k$-shapes and $\mu/\lambda$ is a cover-type string.  It is
maximal (resp. reverse-maximal) if $\mu/\lambda$ is a maximal string (resp. reverse-maximal string).
\end{definition}

A standard $k$-shape tableau of shape $\lambda$ is a sequence of $k$-shapes
$$
\emptyset= \lambda^{(0)} ,\lambda^{(1)},\dots,\lambda^{(n-1)},\lambda^{(n)}=\lambda
$$
such that $\lambda^{(i)}/\lambda^{(i-1)}$ is a cover for all $i=1,\dots,n$.
A standard $k$-shape tableau is maximal (resp. reverse-maximal) if every
cover composing the tableau is maximal (resp. reverse-maximal).  A standard
$k$-shape tableau is naturally associated to a filling of $\lambda$ such that
letter $i$ occupies the cells $\lambda^{(i)}/\lambda^{(i-1)}$.  Given a $k$-shape
tableau $T$ of $n$ letters, we let $T_i$ be the subtableau of $T$ obtained
by removing all letter $i+1,\dots,n$ from $T$.

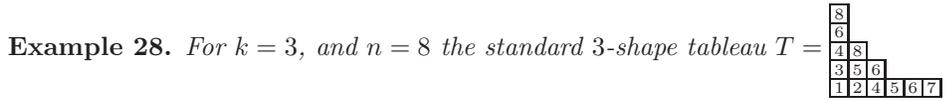
\begin{example}
For $k=3$, and $n=8$ the standard $3$-shape tableau
$T = {\tiny{\tableau[scY]{8|6|4,8|3,5,6|1,2,4,5,6,7}}}$

\noindent corresponds to the sequence of $3$-shapes

$$ \lambda^{(1)} = {\tiny{\tableau[scY]{\bullet}} } \qquad \quad  \lambda^{(2)} = {\tiny{\tableau[scY]{,\bullet}} }\qquad \quad  \lambda^{(3)} = {\tiny{\tableau[scY]{ \bullet|,}}} \qquad \quad \lambda^{(4)} = {\tiny{\tableau[scY]{\bullet| |\bl ,,\bullet}} }
 \qquad \quad \lambda^{(5)} = {\tiny{\tableau[scY]{|,\bullet |\bl ,\bl ,,\bullet}}}
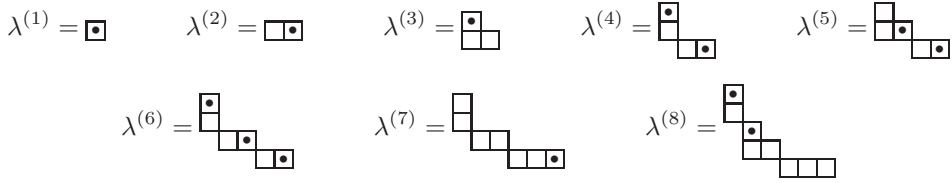
$$
$$
\lambda^{(6)} = {\tiny{\tableau[scY]{ \bullet||\bl ,,\bullet|\bl,\bl,\bl ,,\bullet }}} \qquad \quad
\lambda^{(7)} = {\tiny{\tableau[scY]{||\bl,,|\bl,\bl,\bl,,,\bullet}}}
\qquad \quad
\lambda^{(8)} = {\tiny{\tableau[scY]{\bullet||\bl ,\bullet|\bl,,|\bl,\bl,\bl,,,}}}
$$
\noindent It can be checked easily that $\lambda^{(i)}/\lambda^{(i-1)}$
(represented by $\bullet$ in the diagrams) is a cover
for all $i$ from 1 to 8.

\end{example}

For a given $k$, the following proposition allows to connect sequences of maximal and reverse-maximal covers to standard $k-1$-tableaux and $k$-tableaux respectively.
\begin{proposition}[\cite{LLMS2}]  \label{propktab}
A sequence $\emptyset=\lambda^{(0)} \subseteq \lambda^{(1)} \subseteq \cdots \subseteq \lambda^{(N)}=\lambda$ is a standard $k$-tableau if and only if
$\lambda$ is a $k+1$-core and
$\lambda^{(i)}/\lambda^{(i-1)}$ is a reverse-maximal cover for all $i$.  Similarly, a sequence $\emptyset=\lambda^{(0)} \subseteq \lambda^{(1)} \subseteq \cdots \subseteq \lambda^{(N)}=\lambda$ is a $k-1$-tableau if and only if
$\lambda^{(i)}/\lambda^{(i-1)}$ is a maximal cover for all $i$.
\end{proposition}

\section{Pushout algorithm in the standard case} \label{SecPush}

The main result of \cite{LLMS2} is the construction of a bijection
between pairs $(S,[\bp])$ and $(\tilde S, [\tilde \bp])$,
where $S=\mu/\lambda$ (resp. $\tilde S=\omega/\delta$) is a certain
reverse-maximal strip (resp. maximal strip),
where $\bp$ (resp. $\tilde \bp$) is a path in the poset of $k$-shapes
from $\lambda$ to $\delta$ (resp. $\mu$ to $\omega$), and where
$[\bp]$ denotes the equivalence class of
the path $\bp$.  The bijection can be described diagrammatically by the
commuting diagram
\begin{equation}
\begin{diagram}
\node{\lambda} \arrow{s,l}{S}
 \arrow{e,t}{[\bp]} \node{\delta} \arrow{s,r,..}{\tilde S}\\
\node{\mu} \arrow{e,t,..}{[\tilde \bp]}
\node{\omega}
\end{diagram}
\end{equation}
In the standard case, that is, when the strips $S$ and $\tilde S$ are of rank 1, the bijection is between pairs $(c,[\bp])$ and $(\tilde c, [\tilde \bp])$
where  $c=\mu/\lambda$ (resp. $\tilde c$) is
a reverse-maximal cover (resp. maximal cover)
\begin{equation} \label{bijecstand}
\begin{diagram}
\node{\lambda} \arrow{s,l}{c}
 \arrow{e,t}{[\bp]} \node{\delta} \arrow{s,r,..}{\tilde c}\\
\node{\mu} \arrow{e,t,..}{[\tilde \bp]}
\node{\omega}
\end{diagram}
\end{equation}
The map  $(c,[\bp]) \mapsto (\tilde c, [\tilde \bp])$ is given by a certain pushout algorithm.
We will now describe the canonical form of the pushout algorithm, which given the pair $(c,\bp)$ outputs a pair
$(\tilde c, \tilde \bp)$.  Note that the inverse algorithm associates to
$(\tilde c, \tilde \bp)$ a pair $(c, \bp')$, where
$[\bp']=[\bp]$.  This explains the need to work modulo
equivalences.

The basic ingredients of the
algorithm are the maximization (below and above) of a cover $c$ and the pushout of the pair
$(c,m)$ when $c$ is a maximal cover and $m$ is a move.  Repeated application
of the following three steps will produce the pair
$(\tilde c, \tilde \bp)$ out of $(c, \bp)$.

\noindent {\bf Step 1 (maximization below).}
If the cover
$c=\mu/\lambda=\{a_1,\dots,a_\ell\} $ can be continued below, then
let $c'=c\cup m=(\mu \cup m)/\lambda$,
where $m=\{ a_{\ell+1},\dots,a_{\ell+n}\}$ is the
longest sequence of contiguous addable corners of $\lambda$ such that
$a_{\ell+1}$ is contiguous to $a_{\ell}$ and below it.  In this case,
it is immediate that $c'$ is a cover and that
$m$ is a row-type string.  Moreover, it can be shown \cite{LLMS2}
that $m$ is in fact a row move from $\mu$ to $\mu \cup m$.
Diagrammatically, this gives
\begin{equation}
\begin{diagram}
\node{\lambda} \arrow{s,l}{c}
 \arrow{e,t}{\emptyset} \node{\lambda} \arrow{s,r,..}{c'}\\
\node{\mu} \arrow{e,t,..}{m}
\node{\mu \cup m}
\end{diagram}
\end{equation}

\begin{example} Here is an example of maximization below when  $k=4$
(the cover and its
maximization below are highlighted with dots):
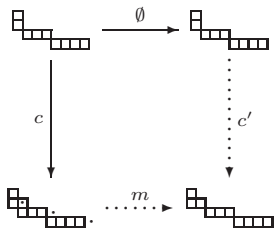
\begin{equation*}
 \begin{diagram}
\node{\tableau[pby]{ \\ \\ \bl& & & \\ \bl& \bl& \bl& \bl& & & &  }} \arrow{s,l}{c}
 \arrow{e,t}{\emptyset} \node{\tableau[pby]{ \\ \\ \bl& & & \\ \bl& \bl& \bl& \bl& & & &  }} \arrow{s,r,..}{c'}\\
\node{\tableau[pby]{ \\ & \cdot \\ \bl& & & & \bl {\, \cdot }\\ \bl& \bl& \bl& \bl& & & &  & \bl {\, \cdot  }}} \arrow{e,t,..}{m}
\node{\tableau[pby]{ \\ &  \\ \bl& \bl& & &   \\ \bl& \bl& \bl& \bl& \bl& & & &  }}
\end{diagram}
\end{equation*}

\end{example}

{\bf Step 2 (maximization above).} If the cover $c$
cannot be continued below but can be continued above,  then
let $c'=c\cup m=(\mu \cup m)/\lambda$,
where $m=\{ a_{-n},\dots,a_{0}\}$ is the
longest sequence of contiguous addable corners of $\lambda$ such that
$a_{0}$ is contiguous to $a_{1}$ and above it.  It is again immediate that $c'$ is a cover and that  $m$ is a column-type string.  Moreover, it is
shown \cite{LLMS2} that
$m$ is in fact a column move from $\mu$ to $\mu \cup m$.
Diagrammatically, this gives
\begin{equation}
\begin{diagram}
\node{\lambda} \arrow{s,l}{c}
 \arrow{e,t}{\emptyset} \node{\lambda} \arrow{s,r,..}{c'}\\
\node{\mu} \arrow{e,t,..}{m}
\node{\mu \cup m}
\end{diagram}
\end{equation}

\begin{example} Here is an example of a maximization above when $k=4$ (the
cover and its maximization above are highlighted with dots):
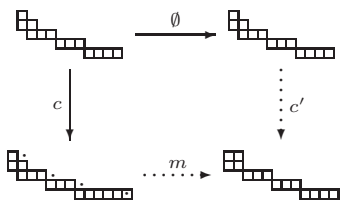
\begin{equation*}
\begin{diagram}
\node{\tableau[pby]{ \\ &\\ \bl& & & \\ \bl& \bl& \bl& \bl& & & \\ \bl& \bl& \bl& \bl& \bl& \bl& \bl& & & &  }} \arrow{s,l}{c}
 \arrow{e,t}{\emptyset} \node{\tableau[pby]{ \\ &\\ \bl& & & \\ \bl& \bl& \bl& \bl& & & \\ \bl& \bl& \bl& \bl& \bl& \bl& \bl& & & &  }} \arrow{s,r,..}{c'}\\
\node{\tableau[pby]{ & \bl {\, \cdot} \\ &\\ \bl& & & & \bl {\, \cdot}\\ \bl& \bl& \bl& \bl& & & & \bl {\, \cdot}\\ \bl& \bl& \bl& \bl& \bl& \bl& \bl& & & & & &\cdot }} \arrow{e,t,..}{m}
\node{\tableau[pby]{ &  \\ & \\ \bl& \bl& & &   \\ \bl& \bl& \bl& \bl& \bl& & &  \\ \bl& \bl& \bl& \bl& \bl& \bl& \bl& \bl& & & &    }}
\end{diagram}
\end{equation*}
\end{example}

{\bf Step 3 (maximal pushout).}
If $c$ is a maximal cover and $m$ is any move then the maximal pushout of
 $(c,m)$ produces the pair
$(\tilde c, \tilde m)$ (which we will describe explicitly below), where
$\tilde m$ is a move:
\begin{equation}
\begin{diagram}
\node{\lambda} \arrow{s,l}{c}
 \arrow{e,t}{m} \node{\nu} \arrow{s,r,..}{\tilde c}\\
\node{\mu} \arrow{e,t,..}{\tilde m}
\node{\eta}
\end{diagram}
\end{equation}
We say that the cover $c=\mu/\lambda$ and the row move $m$ from $\lambda$ to $\mu$ are {\it interfering} if $c$ and $m$ do not intersect and $\mu  \cup m$
is not a $k$-shape (that is, if $\cs(\mu)+\Delta_{\cs}(m)$ is not a partition).
Similarly, we say that the cover $c=\mu/\lambda$ and the column move $m$ from $\lambda$ to $\mu$ are interfering if $c$ and $m$ do not intersect and $\mu  \cup m$
is not a $k$-shape (that is, if $\rs(\mu)+\Delta_{\rs}(m)$ is not a partition).

For a set of cells $A$, we let $A_U$ and $A_R$ be the set of
cells obtained by translating the cells of $A$ by one unit respectively
upward and to the right.

When $m$ is a row move, the pushout is of one of the 4 possible types
\cite{LLMS2}.
\begin{enumerate}
\item[I] If $m$ and $c$ do not intersect and are not interfering
then $\tilde m=m$ and $\tilde c =c$.
\item[II]  If $m$ and $c$ do not intersect but are interfering,
then
$\tilde c= c \cup m_{{\rm comp}}$ and $\tilde m= m \cup m_{{\rm      comp}}$,
where $m_{{\rm comp}}$ (the completion of $m$) is the string obtained by translating to the right (by one unit) the rightmost string of $m$.
\item[III] If $c$ and $m$ intersect, and $c$ continues above $m$ (but not below),
then $\tilde c= c \setminus
(c \cap m)$ and $\tilde m= m \setminus (c \cap m)$.
\item[IV] If $c$ and $m$ intersect and $c$ continues above and below $m$,
then $\tilde c= (c \setminus (c \cap m)) \cup (c \cap m)_U$ and
$\tilde m= (m \setminus (c \cap m)) \cup (c \cap m)_U$.
\end{enumerate}
Observe that to the exception of type II, $\eta$ always
corresponds to the union of the cells of $\lambda$, $c$ and $m$
(there is interference when this union is not a $k$-shape).

\begin{example} Here are examples of the 4 possible types of pushout when
$k=4$.  We indicate cells of the moves with $\bullet$ and cells of the covers
by framed boxes.
\begin{enumerate}
\item[{\rm I}] $m$ and $c$ do not intersect and are not interfering
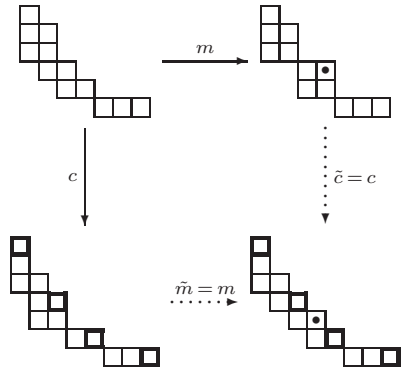
\begin{equation*}
\begin{diagram}
\node{\tiny{\tableau[sby]{ \\ & \\ & \\ \bl& & & \bl {\, } \\ \bl& \bl& &  \\ \bl& \bl& \bl& \bl&  & &   }}} \arrow{s,l}{c}
 \arrow{e,t}{m} \node{\tiny{\tableau[sby]{ \\ & \\ & \\ \bl& \bl& & \bullet \\ \bl& \bl& &  \\ \bl& \bl& \bl& \bl&  & &   }}} \arrow{s,r,..}{\tilde c \, = \, c}\\
\node{\tiny{\tableau[sby]{ \tf \\ \\ & \\ \bl& & \tf \\ \bl& & & \bl {\, }\\ \bl& \bl& \bl& & \tf \\ \bl& \bl& \bl& \bl& \bl& & &   \tf }}} \arrow{e,t,..}{\tilde m \, = \, m}
\node{\tiny{\tableau[sby]{ \tf \\ \\ & \\ \bl& & \tf \\ \bl& \bl& & \bullet \\ \bl& \bl& \bl& &  \tf \\ \bl& \bl& \bl& \bl& \bl& & &   \tf}}}
\end{diagram}
\end{equation*}
\item[{\rm II}]
$m$ and $c$ do not intersect but are interfering
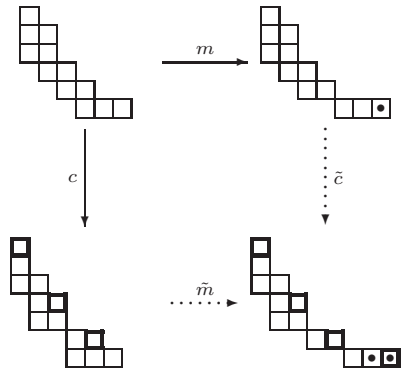
\begin{equation*}
\begin{diagram}
\node{\tiny\tableau[sby]{ \\ & \\ & \\ \bl& &  \\ \bl& \bl& &  \\ \bl& \bl& \bl&   & &  & \bl {\, } }} \arrow{s,l}{c}
 \arrow{e,t}{m} \node{\tiny\tableau[sby]{ \\ & \\ & \\ \bl& &  \\ \bl& \bl& &  \\ \bl& \bl& \bl&  \bl&  & &   \bullet}} \arrow{s,r,..}{\tilde c }\\
\node{\tiny\tableau[sby]{\tf \\ \\ &\\  \bl& &  \tf \\ \bl&  & \\ \bl& \bl& \bl& &  \tf \\ \bl& \bl& \bl& & & & \bl {\, }& \bl {\, } }} \arrow{e,t,..}{\tilde m }
\node{\tiny\tableau[sby]{ \tf \\ \\ &\\  \bl& &  \tf \\ \bl&  & \\ \bl& \bl& \bl& &  \tf \\ \bl& \bl& \bl& \bl& \bl& &\bullet & \tf \bullet  }}
\end{diagram}
\end{equation*}
\item[{\rm III}] $c$ and $m$ intersect, and $c$ continues above $m$ (but not below)
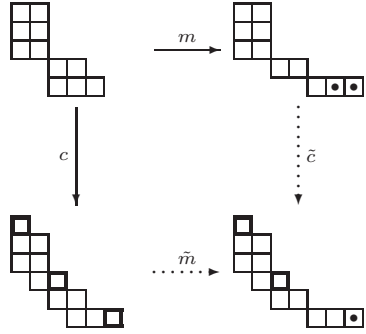
\begin{equation*}
\begin{diagram}
\node{\tiny\tableau[sby]{ & \\ & \\ & \\ \bl& \bl& &  \\ \bl& \bl& & & & \bl {\, } & \bl {\, } }} \arrow{s,l}{c}
 \arrow{e,t}{m} \node{\tiny\tableau[sby]{  & \\ & \\ & \\ \bl& \bl& &  \\ \bl& \bl& \bl& \bl& & \bullet& \bullet}} \arrow{s,r,..}{\tilde c }\\
\node{\tiny\tableau[sby]{ \tf \\  & \\ & \\ \bl& &\tf \\ \bl& \bl& &  \\ \bl& \bl& \bl& & &\tf & \bl {\, } }} \arrow{e,t,..}{\tilde m }
\node{\tiny\tableau[sby]{  \tf \\ & \\ & \\ \bl& & \tf \\ \bl& \bl& &  \\ \bl& \bl& \bl& \bl& & & \bullet }}
\end{diagram}
\end{equation*}
\item[{\rm IV}] $c$ and $m$ intersect and $c$ continues above and below $m$
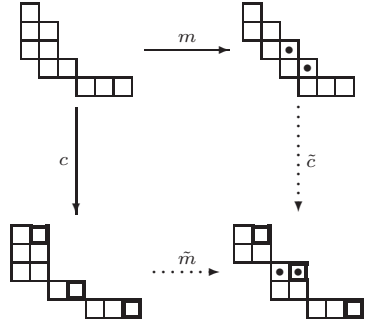
\begin{equation*}
\begin{diagram}
\node{\tiny\tableau[sby]{ \\ & \\ & & \bl {\, }\\ \bl&  &  & \bl {\, }\\ \bl& \bl& \bl& & &  }} \arrow{s,l}{c}
 \arrow{e,t}{m} \node{\tiny\tableau[sby]{  \\ & \\ \bl& & \bullet \\ \bl& \bl&  & \bullet \\ \bl& \bl& \bl& & &  }} \arrow{s,r,..}{\tilde c }\\
\node{\tiny\tableau[sby]{ & \tf \\ & \\ & & \bl {\,} & \bl {\, }\\ \bl& \bl&  & \tf \\ \bl& \bl& \bl& \bl&  & &  \tf }} \arrow{e,t,..}{\tilde m }
\node{\tiny\tableau[sby]{ & \tf \\ & \\ \bl& \bl& \bullet & \tf \bullet \\ \bl& \bl&  &  \\ \bl& \bl& \bl& \bl&   & &\tf }}
\end{diagram}
\end{equation*}
\end{enumerate}
\end{example}

Similarly, if $m$ is a column move the possible types are:
\begin{enumerate}
\item[I] If $m$ and $c$ do not intersect and are not interfering,
then $\tilde m=m$ and $\tilde c =c$.
\item[II] If $m$ and $c$ do not intersect but are interfering,
then
then $\tilde c= c \cup m_{{\rm comp}}$ and $\tilde m= m \cup m_{{\rm      comp}}$,
where $m_{{\rm comp}}$ is the string obtained by translating by one unit upward
the uppermost string of $m$.
\item[III] If $c$ and $m$ intersect, and $c$ continues below $m$ (but not above),
then $\tilde c= c \setminus
(c \cap M)$ and $\tilde M= M \setminus (c \cap M)$.
\item[IV] If $c$ and $m$ intersect and $c$ continues above and
below $m$, then
$\tilde c= (c \setminus (c \cap m)) \cup (c \cap m)_R$ and
$\tilde m= (m \setminus (c \cap m)) \cup (c \cap m)_R$.
\end{enumerate}

Using the three steps repeatedly, the pushout algorithm produces
 from any pair $(c,m)$ a pair $(\tilde c, \bq)$, where $\tilde c$ is a maximal
cover and $\bq$ is a path
\begin{equation}
\begin{diagram}
\node{\lambda} \arrow{s,l}{c}
 \arrow{e,t}{m} \node{\nu} \arrow{s,r,..}{\tilde c}\\
\node{\mu} \arrow{e,t,..}{\bq}
\node{\gamma}
\end{diagram}
\end{equation}
Applying this algorithm for every move in a given path $\bp$,
the pushout algorithm thus produces from any pair $(c,\bp)$
a pair $(\tilde c, \tilde \bp)$, where $\tilde c$ is a maximal cover and $\tilde \bp$ is a path
\begin{equation}
\begin{diagram}
\node{\lambda} \arrow{s,l}{c}
 \arrow{e,t}{\bp} \node{\delta} \arrow{s,r,..}{\tilde c}\\
\node{\mu} \arrow{e,t,..}{\tilde \bp}
\node{\omega}
\end{diagram}
\end{equation}
In the special case where $c$ is a reverse-maximal cover, this corresponds
to the bijection in \eqref{bijecstand}
when equivalences of paths are considered.

\section{Weak bijection in the standard case} \label{Sweakbij}
The main reason to construct bijection \eqref{bijecstand} is to
obtain
the following bijection (weak bijection in the standard case):
\begin{equation}
\begin{split}
{\rm SWTab}_{\lambda}^k  & \longrightarrow  \bigsqcup_{\mu \in {\mathcal C}^k}
{\rm SWTab}_{\mu}^{k-1}
\times \overline{\mathcal P}^k(\lambda,\mu) \\
Q^{(k)}  & \longmapsto    (Q^{(k-1)},[\bp])
\end{split}
\end{equation}
where ${\rm SWTab}_{\lambda}^k$ is the set of standard $k$-tableau (or standard
weak tableau) of shape $\lambda$.  This bijection proceeds as follows.
>From Proposition~\ref{propktab}, a $k$-tableau $Q^{(k)}$ of shape $\lambda$ is a sequence
of reverse-maximal covers $c_1=\lambda^{(1)}/\emptyset,\dots,c_n=\lambda/\lambda^{(n-1)}$. Starting with the pair $(c_1,[\emptyset])$, where $\emptyset$ is
the empty move from the empty partition to itself,
the bijection
\eqref{bijecstand} gives a pair $(\tilde c_1,[\bp_1])$
with $\tilde c_1$ a maximal cover.  Then,
the pair $(c_2,[\bp_1])$ leads to the pair $(\tilde c_2,[\bp_2])$,
where $\tilde c_2$ is again a maximal cover.
 Continuing this way we obtain that $Q^{(k)}$ is in correspondence with
a sequence of maximal covers $\tilde c_1,\dots,\tilde c_n$
and an equivalence class of paths $[\bp]$.  This is illustrated
in the following diagram:
\begin{equation} \label{ktabdiag}
\begin{diagram}
\node{\emptyset} \arrow{s,l}{c_1}
 \arrow{e,t}{\emptyset} \node{\emptyset} \arrow{s,r}{\tilde c_1}\\
\node{\lambda^{(1)}} \arrow{s,l}{c_2}
 \arrow{e,t}{[\bp_1]} \node{\mu^{(1)}} \arrow{s,r}{\tilde c_2}\\
\node{\lambda^{(2)}} \arrow{s,l,..}{}
 \arrow{e,t}{[\bp_2]} \node{\mu^{(2)}} \arrow{s,r,..}{}\\
\node{\lambda^{(n-1)}} \arrow{s,l}{c_{n}}
 \arrow{e,t}{[\bp_{n-1}]} \node{\mu^{(n-1)}} \arrow{s,r}{\tilde c_{n}}\\
\node{\lambda} \arrow{e,t}{[\bp]}
\node{\mu}
\end{diagram}
\end{equation}
The sequence of maximal covers $\tilde c_1,\dots,\tilde c_n$ starting at the empty partition is a certain standard $k-1$-tableau $Q^{(k-1)}$ by
Proposition~\ref{propktab}.  Hence $Q^{(k)}$ is in correspondence with
$(Q^{(k-1)},[\bp])$ as claimed.

\section{Missing bijections and what they would entail} \label{missing}
To give some perspective to the present work, we explain our general approach
to prove the Schur positivity of $k$-Schur functions and dual $k$-Schur functions.  We also describe how
this approach  would provide a combinatorial formula for
the $k$-Schur expansion of
Hall-Littlewood polynomial indexed by $k$-bounded partitions, and how it relates to the atoms of \cite{LLM}.

Certain weak and strong tableaux (a.k.a. $k$-tableaux and dual $k$-tableaux)
related respectively to the weak and strong order on Grassmannian permutations of the affine symmetric group
$\tilde S_{k+1}$ were introduced in \cite{LM:cores,LLMS}.
 These tableaux have a certain weight (just as the usual tableaux do) which tells how many times a given letter appears in the tableau.

The graded $k$-Schur functions (depending on a parameter $t$ and indexed by
$k+1$-cores) are defined as \cite{LLMS}
\begin{equation}
s_{\lambda}^{(k)}(x;t) = \sum_{P} t^{{\rm spin} (P)} x^P
\end{equation}
where the sum is over all strong tableaux $P$ of shape $\lambda$,
where spin is a certain statistic on strong tableaux,
and where $x^P=x_1^{\alpha_1} \cdots x_n^{\alpha_n}$
if the weight of $P$ is $(\alpha_1,\dots,\alpha_n)$.

We now provide a similar definition for the
graded dual $k$-Schur functions (indexed again by
$k+1$-cores).  Let
\begin{equation}
{\mathfrak S}_{\lambda}^{(k)}(x;t) = \sum_{Q} t^{{\charge} (Q)} x^Q
\end{equation}
where the sum is over all weak tableaux $Q$ of shape $\lambda$,
and where we recall that the charge of a $k$-tableau was defined in
Section~\ref{Schargektab}.

Computational evidence suggests the conjecture
that the $k$-Schur functions expand positively into $k'$-Schur
functions for $k'>k$:
\begin{equation}\label{surepos}
s_\mu^{(k)}(x;t) = \sum_{\lambda}
 b_{\mu \lambda }^{(k\to k')}(t)\, s_\lambda^{(k')}(x;t)\,,\qquad\text{for
$b_{\lambda \mu}^{(k\to k')}(t) \in\Z_{\ge0}[t]$.}
\end{equation}
For sufficiently large $k'$, it is known that $k'$-Schur functions and Schur functions coincide.  The conjecture thus implies in particular the Schur positivity of $k$-Schur functions.

In order to prove the previous conjecture, it is sufficient
to understand the case when
$k$ and $k'$ differ by one.  An explicit conjecture concerning this case
was stated in \cite{LLMS2} (and proven in the case $t=1$).
 It relates the coefficients $b_{\mu \lambda }^{(k-1\to k)}(t)$ (the branching coefficients) to certain paths in the
poset of $k$-shapes weighted by charge.  Let
\begin{equation}
b^{(k)}_{\mu\la}(t) := \sum_{[\bp]\in \Patheq^k(\la,\mu)} t^{\charge(\bp)}.
\end{equation}

\begin{conjecture} \label{CJ:branch}
For all $\la \in\Core^{k+1}$ and $\mu\in\Core^k$, the special case
$k \mapsto k-1$ and $k' \mapsto k$ of \eqref{surepos} is given by
\begin{equation}
 b_{\mu \lambda }^{(k-1\to k)}(t)  = b^{(k)}_{\mu\la}(t)
\end{equation}
\end{conjecture}

It is also conjectured that
the dual $k$-Schur functions
expand positively  into $k'$-Schur functions, but this time for $k'< k$:
\begin{equation}\label{dualsurepos}
{\mathfrak S}_\mu^{(k)}(x;t) = \sum_{\lambda}
 b_{\lambda \mu }^{(k'\to k)}(t)\, {\mathfrak S}_\lambda^{(k')}(x;t) \mod I_{k'}
\end{equation}
where $I_{k'}$ is the ideal generated by the monomial symmetric functions
$m_{\rho}$ such that $\rho_1 > k'$.  We stress that
the conjecture makes the stronger claim that the coefficients are the same (up to transposition) as those in the decompositions \eqref{surepos}.   This conjecture,
also shown to hold when $t=1$ in \cite{LLMS2},
would follow again from the case where $k$ and $k'$ differ by one.
\begin{conjecture} \label{CJ:branchdual}
For all $\mu \in\Core^{k+1}$ and $\la\in\Core^k$, the special case
$k' \mapsto k-1$ of \eqref{surepos} is given by
\begin{equation} \label{E:gradedbranchingdual}
  b_{\la \mu}^{(k-1\to k)}(t)  = b^{(k)}_{\la \mu}(t)
\end{equation}
\end{conjecture}

We now describe an approach to prove Conjectures~\ref{CJ:branch}
and \ref{CJ:branchdual}.  Let ${\rm WTab}_{\lambda}^k$ and ${\rm STab}_{\lambda}^k$
stand respectively for the set of $k$-tableaux and dual $k$-tableaux of shape $\lambda$.  As we will see, it would be immediate that the conjectures hold if one could find two weight-preserving bijections (respectively the weak and strong bijections)
\begin{equation} \label{bij1}
\begin{split}
{\rm WTab}_{\lambda}^k  & \longrightarrow  \bigsqcup_{\mu \in {\mathcal C}^k}
{\rm WTab}_{\mu}^{k-1}
\times \overline{\mathcal P}^k(\lambda,\mu) \\
Q^{(k)}  & \longmapsto    (Q^{(k-1)},[\bp])
\end{split}
\end{equation}
and
\begin{equation}\label{bij2}
\begin{split}
{\rm STab}_{\lambda}^{k-1}  & \longrightarrow  \bigsqcup_{\mu \in {\mathcal C}^{k+1}}
{\rm STab}_{\mu}^{k}
\times \overline{\mathcal P}^k(\mu,\lambda) \\
P^{(k-1)}  & \longmapsto    (P^{(k)},[\bp])
\end{split}
\end{equation}
such that
\begin{equation}\label{cond1}
{\rm ch}(Q^{(k)})= {\rm ch}(\bp) + {\rm ch}(Q^{(k-1)})
\end{equation}
and
\begin{equation}\label{cond2}
{\rm spin}(P^{(k-1)})= {\rm ch}(\bp) + {\rm spin}(P^{(k)})
\end{equation}
We should add that in the weak bijection
the weight of $Q^{(k)}$ in \eqref{bij1} needs to be $k-1$-bounded
(that is no letter can occur more than $k-1$ times).
The weak bijection
was constructed in \cite{LLMS2}.  The aim of this article is to show
(Theorem~\ref{theo}) that
it satisfies \eqref{cond1}
in the standard case (we will discuss in the conclusion why the semi-standard
case is still out of reach).  The obtention of
the strong bijection \eqref{bij2}
satisfying \eqref{cond2} is a wide open problem.

We now show how the weak bijection implies
Conjecture~\ref{CJ:branchdual} if it satisfies \eqref{cond1}
(the proof that the strong bijection  implies
Conjecture~\ref{CJ:branch} if it satisfies \eqref{cond2} is practically
identical and will thus be omitted).   We have
from \eqref{bij1}  and \eqref{cond1} that
\begin{equation}
 \sum_{Q^{(k)} \in {\rm WTab}_{\lambda}^k} t^{{\rm ch}(Q^{(k)})} x^{Q^{(k)}}
 =\sum_{\mu} \sum_{[\bp]\in \Patheq^k(\la,\mu)} \sum_{Q^{(k-1)} \in {\rm WTab}_{\mu}^{k-1}}
t^{{\rm ch}(Q^{(k-1)}) + {\rm ch}(\bp) } x^{Q^{(k-1)}} \mod I_{k-1}
\end{equation}
where the equality only holds modulo
$I_{k-1}$ since the leftmost sum is only over $k$-tableaux $Q^{(k)}$
of $k-1$-bounded weight.  Therefore, it immediately follows that
\begin{equation}
{\mathfrak S}_{\lambda}^{(k)}(x;t) = \sum_\mu
\sum_{[\bp]\in \Patheq^k(\la,\mu)}
t^{{\rm ch}(\bp) } {\mathfrak S}_{\mu}^{(k-1)}(x;t) \mod I_{k-1}
\end{equation}
which is equivalent to Conjecture~\ref{CJ:branchdual}.

Another interesting consequence of the weak and strong
bijections would be the following explicit
expression for the $k$-Schur expansion of Hall-Littlewood polynomials
indexed by $k$-bounded partitions:
\begin{equation} \label{HLkschur}
H_{\lambda}(x;t) = \sum_{Q^{(k)}} t^{{\rm ch}(Q^{(k)})} s_{{{\rm sh} (Q^{(k)})}}^{(k)}(x;t)
\end{equation}
where the sum is over all $k$-tableaux of weight $\lambda$.
This formula generalizes a well-known result of
 Lascoux and Sch\"utzenberger providing
a $t$-statistic on tableaux for the Kostka-Foulkes polynomials \cite{LScharge}.

This is shown in the following way.  Let $n=|\lambda|$.
If we iterate  the weak bijection
\begin{equation}
Q \mapsto (Q^{(n-1)},[\bp_{n}]), \, \quad
 Q^{(n-1)} \mapsto (Q^{(n-2)},[\bp_{n-1}]), \quad
\dots \, , \, \quad Q^{(k+1)} \mapsto (Q^{(k)},[\bp_{k+1}])
\,
\end{equation}
 we can establish a bijective correspondence
\begin{equation} \label{bijn}
\begin{split}
{\rm WTab}_{\lambda}^n  & \longrightarrow  \bigsqcup_{
\mu \in {\mathcal C}^{k+1}}
{\rm WTab}_{\mu}^{k}
\times \overline{\mathcal P}^{n \to k}(\lambda,\mu)
\\
Q  & \longmapsto  (Q^{(k)},[\bp_{n}],\dots,[\bp_{k+1}])
\end{split}
\end{equation}
where $\overline{\mathcal P}^{n \to k}(\lambda,\mu)$ is the
set of all sequences
$[\bp_{n}],\dots,[\bp_{k+1}]$ such that
$$ \bp_n \in {\mathcal P}^{n}(\lambda,\mu^{(n-1)}),  \bp_{n-1} \in {\mathcal P}^{n-1}(\mu^{(n-1)},\mu^{(n-2)}),
\dots,
\bp_{k+1} \in
{\mathcal P}^{k+1}(\mu^{(k+1)},\mu)$$
Moreover, the previous correspondence satisfies
$$
{\rm ch}(Q)= {\rm ch}(Q^{(k)})+{\rm ch}( \bp_n) + \cdots + {\rm ch}( \bp_{k+1}) \,
$$
which implies that
\begin{equation}
H_{\lambda}(x;t) = \sum_{Q} t^{{\rm ch}(Q)} s_{{{\rm sh} (Q)}} =
 \sum_{Q^{(k)}}  t^{{\rm ch}(Q^{(k)})} \sum_{\mu}
\sum_{[\bp] \in \overline{\mathcal P}^{n \to k}(\mu,{\rm sh}(Q^{(k)}))}
 t^{{\rm ch}( \bp)} s_{\mu}
\end{equation}
where for short we use $[\bp]$ for $[\bp_n],\dots,[\bp_{k+1}]$, and
$\charge(\bp)=\charge(\bp_n)+\cdots +\charge(\bp_{k+1})$.
Equation \eqref{HLkschur} then follows since repeated applications of Conjecture~\ref{CJ:branch} gives
\begin{equation}
\sum_{\mu}\sum_{[\bp]  \in \overline{\mathcal P}^{n \to k}(\mu,{\rm sh}(Q^{(k)}))}
 t^{{\rm ch}(\bp)} s_{{{\rm sh} (Q)}}(x) = s_{{\rm sh}(Q^{(k)})}^{(k)}(x;t)
\end{equation}

Finally, we discuss the connection between the set $\mathcal A_{\mu}^{(k)}$
(called atoms in
\cite{LLM}) and the weak bijection \eqref{bij1}.  When $
Q \mapsto (Q^{(k)},[\bp_{n}],\dots,[\bp_{k+1}])$, we say that $Q^{(k)}$ is the $k$-tableau associated to $Q$ or that $Q$ maps to $Q^{(k)}$.
It was conjectured in \cite{LLMS2} that the set $\mathcal A^{(k)}_{\mu}$ appearing in \eqref{eqatom}
is given by the set
of tableaux that map to a certain $k$-tableaux $T_{\mu}^{(k)}$:
\begin{conjecture}  Let $\rho$ be the  unique element of $\Core^{k+1}$
such that $\rs(\rho)=\mu$, and let
$T_{\mu}^{(k)}$ be the unique $k$-tableau of
weight $\mu$ and shape $\rho$ (see \cite{LM:cores}).
Then
\begin{equation}
\mathcal A_\mu^{(k)} = \left\{ T {\rm ~of~weight~} \mu \, \big| \, T_{\mu}^{(k)}
{\rm ~is~the~} k{\text -}{\rm tableau~associated~to~} T \right\}
\,.
\end{equation}
\end{conjecture}
Our results would imply that if $Q_1^{(k)}$ and $Q^{(k)}_2$
are two
$k$-tableaux of the same shape, then the sets of tableaux $\mathcal A$ and $\mathcal B$
that respectively map to $Q_1^{(k)}$ and $Q^{(k)}_2$ are each in correspondence with the elements $[\bp_{n}],\dots,[\bp_{k+1}]$
of $\overline{\mathcal P}^{n \to k}(\lambda,\mu)$.  Hence
\begin{equation}
t^{-\charge(Q_1^{(k)})}
\sum_{T\in \mathcal A} t^{\charge(T)} s_{{\rm shape}(T)} = t^{-\charge(Q_2^{(k)})}
\sum_{T'\in \mathcal B} t^{\charge(T')} s_{{\rm shape}(T')}
\end{equation}
given that
\begin{equation}
{\rm ch}(T)- {\rm ch}(Q_1^{(k)})= {\rm ch}(T')- {\rm ch}(Q_2^{(k)})
={\rm ch}( \bp_n) + \cdots + {\rm ch}( \bp_{k+1}) \, \end{equation}
The two symmetric functions
$\sum_{T\in \mathcal A} t^{\charge(T)} s_{{\rm shape}(T)}$ and
$\sum_{T'\in \mathcal B} t^{\charge(T')} s_{{\rm shape}(T')}$ thus only differ
by a power of $t$.
In the language of \cite{LLM}, these are instances of copies of atoms (which
in \cite {LLM} where only conjectured to exist).

\section{(Co)charge of a standard $k$-shape tableau} \label{Schargekshape}

We now generalize the notions of charge and cocharge
to standard $k$-shape tableau.
We will then establish a relation
between charge and cocharge (see Proposition~\ref{propchargecocharge})
that will prove very useful in the proof
of the compatibility between (co)charge and the weak bijection.

\subsection{$k$-connectedness}
Let $r$ and $r'$ (with $r>r'$)
be rows of the $k$-shape $\lambda$ that each have an
addable corner.
We say that
$r'$ is the $k$-connected row below row $r$ (or simply that $r$ and $r'$
are $k$-connected rows) if $r'$ is the lowest row such
that the distance between
the addable corners in row $r$ and $r'$ is not larger than $k+1$.
If the distance between the addable corners in rows $r$ and $r'$ is $k$ or $k+1$
then $r$ and $r'$ are said to be contiguously connected.
We say that $r_1,\dots,r_m$ form a sequence of
$k$-connected rows of length $m$ if $r_{i+1}$ is the $k$-connected row below
row $r_i$ for all $i=1,\dots,m-1$.

\begin{example}
Let $k=5$  and $\lambda = \tableau[pcy]{ \\ & \\ \bl& & & \\ \bl& \bl& & & & \\ \bl& \bl& \bl& \bl& & & &  \\ \bl& \bl& \bl& \bl& \bl& \bl& \bl& \bl& & & & }$.
The pairs of 5-connected rows are: 7 and 5; 5 and 3; 3 and 2; 2 and 1; 6 and 4; 4 and 2. Moreover  rows 4 and 2 are contiguously connected,  as are rows 5 and 3.
Rows $7,5,3,1$ form a sequence of
5-connected rows of length 4.
Observe that two distinct rows can have the same
 $k$-connected row below:
for instance row 2 is the 5-connected row below rows 3 and 4.
\end{example}

Let $r$ and $r'$ (with $r>r'$)
be rows of the $k$-shape $\lambda$ that each have an
addable corner.  Define  $[r,r']_k$ to be equal to the length
of the longest sequence of $k$-connected rows $r_1, r_2, \dots, r_m$
such that $r_1=r$ and $r_m \geq r'$.   We define $(r,r']_k$, $[r,r')_k$ and $(r,r')_k$ in the same fashion (not counting row $r$ or $r'$ according to whether the interval is closed or open).

\subsection{Definition of charge and cocharge}\label{subchargeco}
 The charge of a $k$-shape tableau $T$ on $N$ letters is
\begin{equation}
{\charge} (T) = \sum_{n=1}^N \charge(n)
\end{equation}
where $\charge(1)=0$, and where $\charge(n)$ for $n>1$ is
defined recursively
in the following way.  Let $r$ be the row above that of
$n_-^\uparrow$ and $r'$ be the row of $n^\uparrow$. Then
\begin{equation} \label{defcharge2}
   \charge(n) =   \left \{ \begin{array}{l@{\qquad }l} \\
    \charge(n-1) +[r,r')_k  & {\rm if} \quad  r\geq r' \\ \\
    \charge(n-1) -(r',r]_k & {\rm if} \quad r < r'
   \end{array}
    \right .
\end{equation}
where $[r,r')_k$ and $(r',r]_k$ are calculated using the $k$-shape corresponding to the shape of $T_{n-1}$.  We also stress that $[r,r')_k=0$ if $r=r'$.

\begin{example} \label{exa36}
Let $k=4$ and \quad $T = {\small{\tableau[scY]{9|7|4,6,9|3,5,7|1,2,4,6,8,9}}}$.

\noindent The sequence of $k$-shapes with the corresponding rows $r$ and $r'$
are

$${\small{\tableau[scY]{ \bl r|,\bl r'}} } \qquad  \qquad   {\small{\tableau[scY]{\bl r\phantom{'}, \bl r'|,}} }\qquad  \qquad   {\small{\tableau[scY]{\bl r\phantom{'},\bl r'| | ,}}} \qquad  \qquad  {\small{\tableau[scY]{\bl r| |,\bl r'|\bl ,,}}} \qquad  \qquad    {\small{\tableau[scY]{,\bl r\phantom{'}, \bl r'|,|\bl ,,}}} $$

$$ {\small{\tableau[scY]{\bl r\phantom{'},\bl r' |,|,|\bl ,\bl ,,}}} \qquad  \qquad   {\small{\tableau[scY]{\bl r\phantom{'} | |,|\bl ,,|\bl,\bl,,,\bl r'}}} \qquad  \qquad
{\small{\tableau[scY]{\bl r'||,|\bl,,,\bl r\phantom{'}  |\bl,\bl,,,}}}
$$

\noindent Therefore
 $\charge(1) = 0 ; \quad \charge(2) = \charge(1) + [2,1)_4=1; \quad \charge(3) = \charge(2) - (2,2]_4=1; \quad \charge(4) = \charge(3) - (3,3]_4=1; \quad \charge(5) = \charge(4) + [4,2)_4=2; \quad \charge(6) = \charge(5) - (3,3]_4=2; \quad \charge(7) = \charge(6) - (4,4]_4=2; \quad \charge(8) = \charge(7) + [5,1)_4=4; \quad \charge(9) = \charge(8) - (5,2]_4=3 $.  This gives
$\charge(T) = \sum_{n=1}^9 \charge(n)= 0+1+1+1+2+2+2+4+3 = 16$.
 \end{example}

Similarly, the cocharge of a $k$-shape tableau $T$ on $N$ letters is
\begin{equation}
{\cocharge} (T) = \sum_{n=1}^N \cocharge(n)
\end{equation}
where $\cocharge(1)=0$, and where $\cocharge(n)$ for $n>1$ is defined recursively
in the following way.  Let $r$ be the row above that of
$n_-^\downarrow$ and $r'$ be the row of $n^\downarrow$.  Then
\begin{equation}
  \cocharge(n)=   \left \{ \begin{array}{l@{\qquad }l} \\
    \cocharge(n-1)-(r,r')_k   & {\rm if} \quad  r>r' \\ \\
    \cocharge(n-1) + [r',r]_k & {\rm if} \quad r \leq r' \\ \\
   \end{array}
    \right .
\end{equation}
where $(r,r')_k$ and $[r',r]_k$ are again calculated using the $k$-shape corresponding to the shape of $T_{n-1}$.

\begin{example}\label{exa37} Using the same tableau as in Example~\ref{exa36},
we get the following sequence of $k$-shapes (with rows $r$ and $r'$ identified)

$${\tiny{\tableau[scY]{ \bl r|,\bl r'}} } \qquad  \quad   {\tiny{\tableau[scY]{\bl r\phantom{'}, \bl r'|,}} }\qquad  \quad   {\tiny{\tableau[scY]{\bl r\phantom{'}| | ,,\bl r'}}} \qquad  \quad  {\tiny{\tableau[scY]{ |,\bl r\phantom{'},\bl r'|\bl ,,}}} \qquad  \quad    {\tiny{\tableau[scY]{,\bl r\phantom{'}|,|\bl ,,,\bl r'}}} \qquad \quad {\tiny{\tableau[scY]{,|,,\bl r\phantom{'}, \bl r'|\bl ,\bl ,,}}} \qquad  \quad   {\tiny{\tableau[scY]{ |,,\bl r\phantom{'}|\bl ,,|\bl,\bl,,,\bl r'}}} \quad  \qquad
{\tiny{\tableau[scY]{|,|\bl,,,\bl r\phantom{'}  |\bl,\bl,,,,\bl r'}}}
$$

\noindent We thus have $\cocharge(1) = 0 ; \quad \cocharge(2) = \cocharge(1) - (2,1)_4=0; \quad \cocharge(3) = \cocharge(2) + [2,2]_4=1; \quad \cocharge(4) = \cocharge(3) - (3,1)_4=1; \quad \cocharge(5) = \cocharge(4) + [2,2]_4=2; \quad \cocharge(6) = \cocharge(5) - (3,1)_4=2; \quad \cocharge(7) = \cocharge(6) + [2,2]_4=3; \quad \cocharge(8) = \cocharge(7) - (3,1)_4=3; \quad \cocharge(9) = \cocharge(8) - (2,1)_4=3$.  Therefore,
$\cocharge(T) = \sum_{n=1}^9 \cocharge(n)= 0+0+1+1+2+2+3+3+3 = 15$.
\end{example}

We now show that the definition of (co)charge of a $k$-shape
tableau actually extends that of (co)charge of a $k$-tableau.
\begin{proposition} Let $T$ be a standard $k$-shape tableau that
is also a $k$-tableau.  Then definitions \eqref{defcharge1} and \eqref{defcharge2} of charge coincide (and similarly for cocharge).
\end{proposition}
\begin{proof}  Let $r$ and $r'$ be respectively the row above
that of $n_-^\uparrow$ and the row of $n^\uparrow$.  Let also
$e$ and $e'$ be the residues of $n^\uparrow$
and $n_-^\uparrow$ respectively.

Suppose that $r \geq r'$.  Then we need to show that
\begin{equation}
[r,r')_k= {\rm diag}_{e'} (n_-^\uparrow, n^\uparrow) +1
\end{equation}
or more simply, that
\begin{equation} \label{chargecharge1}
(r,r')_k= {\rm diag}_{e'} (n_-^\uparrow, n^\uparrow)
\end{equation}
It is an elementary fact of $k$-tableaux (see for instance \cite{LM:cores}) that $e \neq e'$.
This implies that
${\rm diag}_{e'} (n_-^\uparrow, n^\uparrow)=
{\rm diag}_{e'-1} (n_-^\uparrow, n^\uparrow)$.  Since there is a residue
$e'-1$ in row $r$, \eqref{chargecharge1} follows from
Lemma~\ref{lemresidueconnect}.

Now suppose that $r < r'$.  We need to show that
\begin{equation}
(r',r]_k= {\rm diag}_{e} (n^\uparrow, n_-^\uparrow)
\end{equation}
or equivalently, that
\begin{equation}
(r',r-1)_k= {\rm diag}_{e} (n^\uparrow, n_-^\uparrow)
\end{equation}
Since there is a residue $e$ in row $r'$, the equation follows again
from Lemma~\ref{lemresidueconnect}.

The proof when charge is replaced by cocharge
is similar.
\end{proof}

\begin{lemma} \label{lemresidueconnect}
Let $\lambda$ be a $k+1$-core.  Let $r'$ be the $k$-connected row below a certain row $r$ of $\lambda$.  If there is a residue
$e$ on the ground  in row
$r$ of $\lambda$ (that is, there is a cell $b=(i,j)$ of residue $e$ in row
$r$ such that $(i,j) \not \in \lambda$ and $(i-1,j) \in \lambda$),
then there
 is also a residue $e$ on the ground in row
$r'$.  Moreover, there is no diagonal of residue $e$ between rows
$r$ and $r'$.
\end{lemma}
\begin{proof}
Let $b$ be the position of the cell of residue $e$ on the ground in row $r$.
Let $c$ be the position of the uppermost cell of ${\rm Int}^k(\lambda)$
in the column of $b$.  It is easy to see that $c$ lies in row
$r'-1$.  Since $h_c(\lambda) > k$, there is a cell $a$ of residue $e$ in
the row of $c$ (see picture).

 \setlength{\unitlength}{0.25mm}
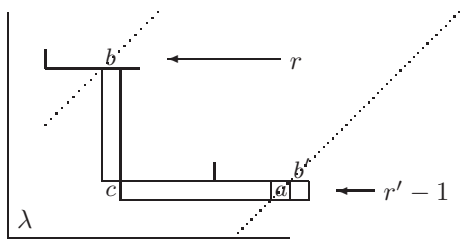
 \begin{picture}(150,130)(-30,-10)

 \put(0,0){\line(1,0){150}}  \put(0,0){\line(0,1){120}}

 \put(20,90){\line(1,0){50}}  \put(20,90){\line(0,1){10}}
 \put(60,20){\line(1,0){100}}  \put(50,30){\line(0,1){60}}           
 \put(50,30){\line(1,0){110}}   \put(60,20){\line(0,1){70}}
 \put(140,20){\line(0,1){10}}  \put(150,20){\line(0,1){10}}
 \put(160,20){\line(0,1){10}}  \put(110,30){\line(0,1){10}}

\put(150,90){$r$}
\put(145,95){\vector(-1,0){60}}

\put(200,20){$r'-1$}
\put(195,25){\vector(-1,0){20}}


\put(52,22){\small $c$}
\put(52,92){\small $b$}
\put(142,22){\small $a$}
\put(152,32){\small $b'$}

\bezier{20}(80,120)(50,90)(20,60)
\bezier{50}(120,0)(180,60)(240,120)

\put(5,5){$\lambda$}

\end{picture}

Observe that by definition of $c$, the cell $b'$ in the picture
does not belong to $\lambda$.  It thus suffices to show that the
cell $b'$ has a cell of
$\lambda$ immediately below.
But this has to be the case since
$\lambda$ is a $k+1$-core
(otherwise we would have $h_c(\lambda) = k+1$).
\end{proof}

\subsection{Relation between charge and cocharge}

Before  establishing the relation between charge and cocharge given in
Proposition~\ref{propchargecocharge} below, we need to prove
a series of technical lemmas that will culminate with
Lemma~\ref{lemmechargecochargen} (which is essentially equivalent to
Proposition~\ref{propchargecocharge}).

\begin{lemma}\label{conec1}
Let $r_1$ y $r_2$ be two contiguously connected rows of a given $k$-shape
$\lambda$, with $r_1>r_2$.  If $s_1$ is a row of $\lambda$ such that
 $s_1>r_1$, then the $k$-connected row  $s_2$ below $s_1$ is such that
$s_2> r_2$.
\end{lemma}
\begin{proof}  The proof can be easily visualized with the following diagram:

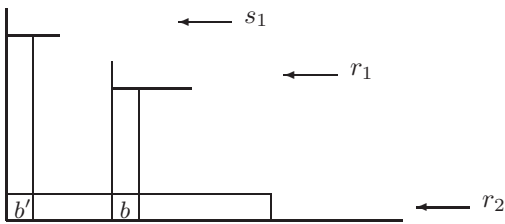
\begin{picture}(200,90)(10,0)

\put(10,10){\line(1,0){150}}  \put(10,10){\line(0,1){80}}
\put(10,80){\line(1,0){20}}  \put(110,10){\line(0,1){10}}
\put(50,60){\line(1,0){30}}  \put(50,10){\line(0,1){60}}
\put(20,10){\line(0,1){70}}  \put(60,10){\line(0,1){50}}
\put(10,20){\line(1,0){100}}

\put(140,65){$r_1$}
\put(190,15){$r_2$}

\put(100,85){$s_1$}

\put(135,65){\vector(-1,0){20}}
\put(185,15){\vector(-1,0){20}}
\put(95,85){\vector(-1,0){20}}

\put(13,11){\small{$b'$}}
\put(53,11){\small{$b$}}


\end{picture}

\noindent Obviously $h_{b'}(\lambda) \geq h_b(\lambda)+2$.  By hypothesis,
$h_b \geq k-1$, which implies $h_{b'}(\lambda) \geq k+1$.  Hence,
from the definition
of $k$-connected rows, we have $s_1> r_2$.
\end{proof}
The next lemma is immediate.
\begin{lemma}\label{conec2}
Let $r_1$ and $r_2$ be two $k$-connected rows of a given $k$-shape
$\lambda$, with $r_1>r_2$.  If $s_1$ is a row of $\lambda$ such that
 $s_1\leq r_1$, then the $k$-connected row  $s_2$ below $s_1$ is such that
$s_2 \leq r_2$.
\end{lemma}

\begin{lemma}\label{conec3}
Let $T$ be any $k$-shape tableau, and let
$\cocharge'(n+1)$ be the
hypothetical cocharge in $T$
of the letter $n+1$  if $n_+^\downarrow$ were
located in the position of $n_+^\uparrow$.
Then $\cocharge'(n+1)-\cocharge(n+1)=\ell$, where $\ell=|c_{n+1}|-1$.

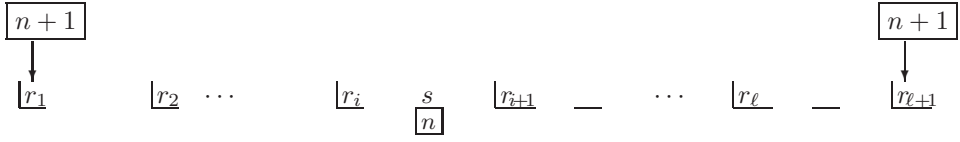
\begin{picture}(200,80)(-40,-10)

\put(-40,20){\line(1,0){10}}  \put(-40,20){\line(0,1){10}}
\put(10,20){\line(1,0){10}}  \put(10,20){\line(0,1){10}}
\put(80,20){\line(1,0){10}}  \put(80,20){\line(0,1){10}}
\put(110,20){\line(1,0){10}}  \put(110,10){\line(0,1){10}}
\put(110,10){\line(1,0){10}}  \put(120,10){\line(0,1){10}}
\put(140,20){\line(1,0){15}}  \put(140,20){\line(0,1){10}}
\put(170,20){\line(1,0){10}}
\put(230,20){\line(1,0){15}}  \put(230,20){\line(0,1){10}}
\put(260,20){\line(1,0){10}}
\put(290,20){\line(1,0){15}}  \put(290,20){\line(0,1){10}}

\put(-38,22){$r_1$}
\put(12,22){$r_2$}
\put(82,22){$r_{i}$}
\put(142,22){$r_{\!i\!+\!1}$}
\put(232,22){$r_{\ell}$}
\put(292,22){$r_{\! \ell + \! 1}$}
\put(112,22){$s$}

\put(112,12){\small $n$}

\put(-45,50){\fbox{$n+1$}}  \put(285,50){\fbox{$n+1$}}

\put(-35,45){\vector(0,-1){15}} \put(295,45){\vector(0,-1){15}}

\put(30,22){$\cdots$} \put(200,22){$\cdots$}
\end{picture}

\end{lemma}
\begin{proof}
Let $T_n$ be  the shape of $\lambda$.
Let $r_1, \dots,r_{\ell+1}$ be the rows where the letter $n+1$ occurs in $T$
($r_1$ and $r_{\ell+1}$ are thus the rows of
$n_+^\uparrow$ and $n_+^\downarrow$ respectively).
These correspond by definition of $c_{n+1}$ to a sequence of
contiguously connected rows of $\lambda$.

Suppose
that $s_1$ is the row of $\lambda$ above that of  $n^\downarrow$.
We will only prove the case where $s_1$ lies between $r_1$
and $r_{\ell+1}$ (the other cases are simplified versions of that case).
We thus suppose that
$s_1$ lies between $r_1$ and $r_{\ell+1}$.
Let $r_i$ be the lowest among $r_1,\dots,r_{\ell+1}$ such that $r_i\geq s_1$.
Let also $s_1,\dots,s_{\ell-i+1}$ be the sequence of
$k$-connected rows weakly below row $s_1$.
>From Lemma~\ref{conec1} and \ref{conec2}, we have
$$
r_i \geq s_1 > r_{i+1} \geq s_2 > \cdots \geq s_{\ell-i} > r_{\ell+1} \geq
s_{\ell-i+1}
$$
and it is then immediate that  $(s_1,r_{\ell+1})_k=(r_i,r_{\ell+1})_k$.

We can now proceed with the proof of the lemma.
We have $\cocharge'(n+1) = \cocharge(n) + [r_1,s_1]_k$ and
$\cocharge(n+1) = \cocharge(n) - (s_1,r_{\ell+1})$. Thus
$\cocharge'(n+1) - \cocharge(n+1) = [r_1,s_1]_k+(s_1,r_{\ell+1})_k$.
But $[r_1,s_1]_k=[r_1,r_i]_k$ and,
as we have seen, $(s_1,r_{\ell+1})_k= (r_i,r_{\ell+1})_k$.  This implies
that $\cocharge'(n+1) - \cocharge(n+1) = [r,r_{\ell+1})_k = \ell$.
\end{proof}

The next lemma, which is similar to the previous one, is stated without proof.
\begin{lemma}\label{conec33}
Let $T$ be any $k$-shape tableau, and
 let $\charge'(n)$ be the
hypothetical charge in $T$ of the letter $n$  if $n^\uparrow$ were
located instead in  the position of $n^\downarrow$.
Then $\charge'(n)-\charge(n)=\ell$, where $\ell=|c_n|-1$.
\end{lemma}

\begin{lemma}\label{conec4}
Let $T$ be any $k$-shape tableau, and let $\charge'(n)$ and $\charge'(n+1)$ be the
hypothetical charge in $T$ of the letters $n$ and $n+1$ respectively
if $n^\uparrow$ were
located instead in the position of $n^\downarrow$.
Then $\charge(n+1)=\charge'(n+1)$.
\end{lemma}
\begin{proof}
Let $T_n$ be the shape of $\lambda$.
Let also
$r_1, \dots,r_{\ell+1}$ be the rows where the letter $n$ occurs in $T_n$,
and define $\bar r_i=r_i+1$ for $i=1,\dots,\ell+1$.
Because of the presence of the letter $n$ in rows
$r_1, \dots,r_{\ell+1}$, the $k$-shape $\lambda$ has
addable corners in rows
 $\bar r_1, \dots,\bar r_{\ell+1}$.  Furthermore, since
by definition of $c_n$ the rows  $r_1, \dots,r_{\ell+1}$
form a sequence of contiguously connected rows, we get that
$\bar r_1, \dots,\bar r_{\ell+1}$ are $k$-connected rows below
row $\bar r_1$.   Observe also that
$\bar r_1$ is the row above that of $n^\uparrow$.

Now, let  $s_1$ be the row of $\lambda$ where   $n_+^\uparrow$ lies,
and let  $s_1,\dots,s_{\ell-i+1}$ be the sequence of $k$-connected rows weakly
below row $s_1$.
We will again
only treat the case where $s_1$ is between $\bar r_1$ and $\bar r_{\ell+1}$.
Define $r_i$ to be the lowest among $r_1,\dots,r_{\ell+1}$ such that $r_i\geq s_1$.
Then from Lemmas~\ref{conec1} and \ref{conec2}, we have
$$
r_i \geq s_1 >  r_{i+1} \geq s_2 > \cdots  \geq s_{\ell-i} > r_{\ell+1} \geq
s_{\ell-i+1}
$$
which is equivalent to
$$
\bar r_i > s_1 \geq  \bar r_{i+1} > s_2 \geq \cdots  > s_{\ell-i} \geq \bar r_{\ell+1} >
s_{\ell-i+1}
$$
Observe that  $(s_1,\bar r_{\ell+1}]_k=(\bar r_{i+1},\bar r_{\ell+1}]_k$.

We can now finalize the proof of the lemma.
We have $\charge'(n+1) = \charge'(n) - (s_1,\bar r_{\ell+1}]_k$ and
$\charge(n+1) = \charge(n) + [\bar r_1,s_1)_k$. Thus
$\charge(n+1) - \charge'(n+1) = \charge(n)-\charge'(n)+[\bar r_1,s_1)_k+(s_1,\bar r_{\ell+1}]_k$.
But $[\bar r_1,s_1)_k=[\bar r_1,\bar r_i]_k$ and,
as we have seen, $(s_1,\bar r_{\ell+1}]_k= (\bar r_{i+1},\bar r_{\ell+1}]_k$.  This implies from Lemma~\ref{conec33}
that $\charge(n+1) - \charge'(n+1) = \charge(n)-\charge'(n)+[\bar r_1,\bar r_i]_k + (\bar r_{i+1},\bar r_{\ell+1}]_k = -\ell+\ell=0$.
\end{proof}

The charge and cocharge of the letter $n$ are related in the following way:
\begin{lemma} \label{lemmechargecochargen}
Let $c_{n}$ be the cover corresponding to letter $n$ in
the $k$-shape tableau $T$.  Then the charge and cocharge of the letter $n$ satisfy the relation
\begin{equation} \label{eqchargecocharge}
 \charge(n) = n-\cocharge(n) - |c_{n}|
\end{equation}
\end{lemma}
\begin{proof}  We proceed by induction.  The case $n=1$ is easily seen to hold.
We need to establish that
\begin{equation}
 \charge(n+1) = n+1-\cocharge(n+1) - |c_{n+1}|
\end{equation}
Using the induction hypothesis, it suffices to prove that
\begin{equation}\label{nada1}
  \charge(n+1)- \charge(n)+\cocharge(n+1)- \cocharge(n) = |c_n|- |c_{n+1}|+1
\end{equation}
We first consider the case where
$n_+^\uparrow$ and $n_+^\downarrow $ are above
$n^\uparrow $ and $n^\downarrow $. We can visualize this case with a diagram.
The symbol $\star$ denotes the rows that
are  $k$-connected with the row of $n_+^\uparrow$, while
$\diamond$ denote the rows above the letters $n$.

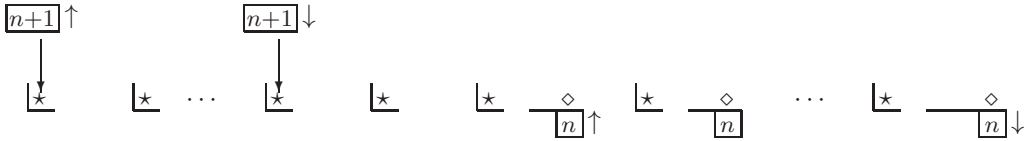
\qquad \begin{picture}(200,80)(0,-10)

\put(10,20){\line(1,0){10}}  \put(10,20){\line(0,1){10}}
\put(50,20){\line(1,0){10}}  \put(50,20){\line(0,1){10}}
\put(100,20){\line(1,0){10}}  \put(100,20){\line(0,1){10}}
\put(140,20){\line(1,0){10}}  \put(140,20){\line(0,1){10}}
\put(180,20){\line(1,0){10}}  \put(180,20){\line(0,1){10}}
\put(240,20){\line(1,0){10}}  \put(240,20){\line(0,1){10}}
\put(330,20){\line(1,0){10}}  \put(330,20){\line(0,1){10}}

\put(12,22){$\star$}
\put(52,22){$\star$}
\put(102,22){$\star$}
\put(142,22){$\star$}
\put(182,22){$\star$}
\put(242,22){$\star$}
\put(332,22){$\star$}

\put(200,20){\line(1,0){20}} \put(210,10){\line(1,0){10}}
 \put(210,10){\line(0,1){10}}  \put(220,10){\line(0,1){10}}
\put(212,12){\small $n$}  \put(222,13){$\uparrow$}

\put(260,20){\line(1,0){20}} \put(270,10){\line(1,0){10}}
 \put(270,10){\line(0,1){10}}  \put(280,10){\line(0,1){10}}
  \put(272,12){\small $n$}

\put(350,20){\line(1,0){30}} \put(370,10){\line(1,0){10}}
 \put(370,10){\line(0,1){10}}  \put(380,10){\line(0,1){10}}
\put(372,12){\small $n$} \put(382,13){$\downarrow$}

\put(212,22){$\diamond$}
\put(272,22){$\diamond$}
\put(372,22){$\diamond$}

\put(70,22){$\cdots$} \put(300,22){$\cdots$}

\put(2,50){\line(1,0){20}} \put(2,50){\line(0,1){10}}
\put(2,60){\line(1,0){20}}  \put(22,50){\line(0,1){10}} \put(24,53){$\uparrow$}                  
\put(15,47){\vector(0,-1){20}} \put(3,52){\small $n$+1}

\put(92,50){\line(1,0){20}} \put(92,50){\line(0,1){10}}
\put(92,60){\line(1,0){20}} \put(112,50){\line(0,1){10}} \put(114,53){$\downarrow$}                  
\put(105,47){\vector(0,-1){20}} \put(93,52){\small $n$+1}

\end{picture}

\noindent Let $r_1, \dots,r_{\ell}$ be the rows of the $\diamond$'s in the picture ($r_1$ and $r_{\ell}$ are the rows of the $\diamond$ above $n^\uparrow$ and  $n^\downarrow$ respectively).
Let also
$s_1$ be the row of the lowest $\star$ such that $s_1 \geq r_1$.
By Lemma~\ref{conec2} the sequence $s_1,\dots,s_{\ell+1}$ of connected row starting from $s_1$ is such
that
$$
s_1 \geq r_1 > s_2 \geq r_2 > \cdots > s_{\ell} \geq r_{\ell} > s_{\ell+1}
$$
We then get
$$( n_+^\uparrow, n^\uparrow ]_k +  ( n^\uparrow \, , \, n^\downarrow ]_k \quad = \, \quad  (  n_+^\uparrow \, ,\, n_+^\downarrow  )_k +  [ n_+^\downarrow\, ,\,  n^\downarrow  ]_k $$
where we identified the cells with their rows.
Therefore, the definition of charge and cocharge imply
\begin{equation}
\charge(n) - \charge(n+1) + |c_n|-1 = |c_{n+1}| -2+\cocharge(n+1) - \cocharge(n)
\end{equation}
which is equivalent to \eqref{nada1}.

We have established that the result holds
when
$n_+^\uparrow$ and $n_+^\downarrow $ are above
$n^\uparrow $ and $n^\downarrow $.   In all the other cases we have either:
\begin{enumerate}
\item[{\bf I)}] $n_+^\uparrow$ or
 $n_+^\downarrow$
 lies between
$n^\uparrow$ and $n^\downarrow $
\item[{\bf II)}] $n^\downarrow$ or
$n^\uparrow$
lies between $n_+^\uparrow $ and $n_+^\downarrow $
\end{enumerate}

Consider case  {\bf I)}.
 From Lemmas~\ref{conec33} and \ref{conec4}, if
we hypothetically shift the position of  $n^\uparrow $ to that
of $n^\downarrow $, we obtain  new $\ch'(n)$
and $\ch'(n+1)$ such that
$\charge'(n)-\charge(n)=|c_{n}|-1$ and  $\ch'(n+1)=\ch(n+1)$.  From
\begin{equation}\label{nadanada}
  \charge'(n+1)- \charge'(n)+\cocharge'(n+1)- \cocharge'(n) = |c_n'|- |c_{n+1}'|+1
\end{equation}
we get using the extra relations $|c_{n}'|=1$, $|c_{n+1}'|=|c_{n+1}|$,
$\cocharge'(n+1)=\cocharge(n+1)$ and
$\cocharge'(n)=\cocharge(n)$, that \eqref{nada1} holds if \eqref{nadanada}
holds.  Note that shifting  $n^\uparrow $ changes its position
relative to either $n_+^\uparrow$ or $n_+^\downarrow$.

Consider case  {\bf II)}.
 From Lemma~\ref{conec3}, if
we hypothetically shift the position of  $n_+^\downarrow $ to that
of $n_+^\uparrow $, we obtain  a new $\cocharge'(n+1)$ such that
$\cocharge'(n+1)-\cocharge(n+1)=|c_{n+1}|-1$.
Using \eqref{nadanada} with  $|c_n'|=|c_n|$,
$|c_{n+1}'|=1$,
$\cocharge'(n)=\cocharge(n)$,
$\charge'(n)=\charge(n)$, and $\charge'(n+1)=\charge(n+1)$, we obtain again that
\eqref{nada1} holds if \eqref{nadanada}
holds.   Observe that shifting  $n_+^\downarrow $ changes its position
relative to  either $n^\uparrow$ or
 $n^\downarrow$.

Applying cases  {\bf I)} and  {\bf II)} again and again,
 we get that the general
case follows from the previously established case
 where $n_+^\uparrow$ and $n_+^\downarrow $ are above
$n^\uparrow $ and $n^\downarrow $.
 \end{proof}

An immediate consequence of Lemma~\ref{lemmechargecochargen} is the
following relation between the charge and the cocharge of a standard
$k$-shape tableau, which generalizes the usual relation between charge and
cocharge of a standard tableau.
\begin{proposition} \label{propchargecocharge}
Let $T$ be a standard $k$-shape tableau of
shape $\lambda$.  Then
\begin{equation}
 \charge(T) = \frac{n(n-1)}{2}-\cocharge(T) - |\Int^k(\lambda)|
\end{equation}
where we recall that $\Int^k(\lambda)$ was defined at the beginning
of Subsection~\ref{SS:kshapes}.
\end{proposition}

\begin{proof} Summing  \eqref{eqchargecocharge} from 1 to $n$, we get
\begin{equation}
 \charge(T) = \frac{n(n+1)}{2}-\cocharge(T) - (|c_1| +\cdots+|c_n|)
\end{equation}
Adding the cover $c_i$ increases by $|c_i|-1$
the number of hooks larger
than $k$. Hence  $|c_1| +\cdots+|c_n|-n=\Int^k(\lambda)$, and the corollary follows.
\end{proof}
\begin{example}
Using the standard $k$-shape tableau $T$ in Examples~\ref{exa36} and \ref{exa37}, we get
\[ \charge(T) = \frac{9(9-1)}{2} - \cocharge(T) - |\Int^{k}(\lambda_9)|
= 36-15-5=16 \]
as wanted.
\end{example}

\section{Compatibility between (co)charge and the weak bijection}
\label{Scompatibility}

Our goal is to show that the weak bijection \eqref{bij1} satisfies the conditions
\begin{equation}
{\cocharge}_k(Q^{(k)})= {\cocharge}(\bp) + {\cocharge}_{k-1}(Q^{(k-1)}) \qquad
{\rm and} \qquad {\charge}_k(Q^{(k)})= {\charge}(\bp) + {\charge}_{k-1}(Q^{(k-1)})
\end{equation}
when $Q^{(k)}$ (resp. $Q^{(k-1)})$ is  standard
$k$-tableau (resp. $k-1$-tableau).  We have emphasized that the (co)charge of $Q^{(k)}$ and $Q^{(k-1)}$ are computed
considering that they are sequences of $k$ and $k-1$-shapes respectively.  The next proposition shows that this distinction is not necessary.
\begin{proposition} \label{propkkmas}
The cocharge (resp. charge) of a standard $k$-tableau $T$
is the same whether it is
considered as a sequence of $k$-shapes or as a sequence of
$k+1$-shapes.  That is,
$$\cocharge_k(T) = \cocharge_{k+1}(T) \qquad {\rm and} \qquad
\charge_k(T) = \charge_{k+1}(T)
$$
if $T$ is a standard $k$-tableau.
\end{proposition}
\begin{proof}
Since a $k+1$-core does not have hooks of length $k+1$, two rows of a
$k+1$-core are $k+1$-connected if and only if they are $k$-connected.
In the case of a $k$-tableau, the corresponding sequence of $k$-shapes
is a sequence of $k+1$-cores.  Hence the result follows immediately from
the definition of cocharge (resp. charge).
\end{proof}

It thus suffices to show that the compatibility between the (co)charge
and the weak bijection (in the standard case) is true  when
the $k-1$-tableau is considered as a sequence of $k$-shapes.
The remainder of this section will be devoted to showing the next proposition,
which from Proposition~\ref{propkkmas} implies Theorem~\ref{theo}.
\begin{proposition} \label{propoprinci}
 Let $Q^{(k)}$ and $Q^{(k-1)}$ be standard $k$ and $k-1$-tableaux respectively such that $Q^{(k)} \longleftrightarrow (Q^{(k-1)},[\bp])$ in the
weak bijection \eqref{bij1}.  Then
\begin{equation}\label{cond111}
{\cocharge}(Q^{(k)})= {\cocharge}(\bp) + {\cocharge}(Q^{(k-1)}) \qquad
{\rm and} \qquad {\charge}(Q^{(k)})= {\charge}(\bp) + {\charge}(Q^{(k-1)})
\end{equation}
where the (co)charge is computed considering that both $Q^{(k)}$ and $Q^{(k-1)}$
are sequences of $k$-shapes.
\end{proposition}
In order to show Proposition~\ref{propoprinci},
we will proceed locally
using the pushout algorithm.
We say that the standard $k$-shape tableaux
$T$ and $U$  differ by moves if there exists a sequence of row moves (resp. column moves)
$m_1,\dots,m_N$ such that the following diagram commutes:
\begin{equation} \label{ktabdiag2}
\begin{diagram}
\node{\emptyset} \arrow{s,l}{c_1}
 \arrow{e,t}{\emptyset} \node{\emptyset} \arrow{s,r}{\tilde c_1}\\
\node{\lambda^{(1)}} \arrow{s,l}{c_2}
 \arrow{e,t}{m_1} \node{\mu^{(1)}} \arrow{s,r}{\tilde c_2}\\
\node{\lambda^{(2)}} \arrow{s,l,..}{}
 \arrow{e,t}{m_2} \node{\mu^{(2)}} \arrow{s,r,..}{}\\
\node{\lambda^{(N-1)}} \arrow{s,l}{c_{N}}
 \arrow{e,t}{m_{N-1}} \node{\mu^{(N-1)}} \arrow{s,r}{\tilde c_{N}}\\
\node{\lambda} \arrow{e,t}{m_N}
\node{\mu}
\end{diagram}
\end{equation}
where $c_1,\dots,c_N$ and $\tilde c_1,\dots, \tilde c_N$ correspond respectively
to $T$ and $U$, and where every commutative square in the diagram corresponds
to one of the 3 steps described in Section~\ref{SecPush}.  The pushout algorithm ensures that one can obtain a sequence
\begin{equation}
Q^{(k)}=T^{(0)}, T^{(1)}, \dots ,T^{(r-1)}, T^{(r)}=Q^{(k-1)}
\end{equation}
where, for all $i$, $T^{(i)}$ and $T^{(i+1)}$ either differ by row moves or column moves.  Therefore, in order to prove Proposition~\ref{propoprinci}, it suffices to show that $T$ and $U$ described in
\eqref{ktabdiag2} are such that
\begin{equation}
{\cocharge}(T)= {\cocharge}(m_N) + {\cocharge}(U)
\qquad {\rm and} \qquad
{\charge}(T)= {\charge}(m_N) + {\charge}(U)
\end{equation}
Using the definition of charge and cocharge, is is straightforward to see that
this is equivalent to proving that
\begin{equation} \label{relcocharge}
\cocharge(N)-\cocharge(\bar N)= {\cocharge}(m_N) - {\cocharge}(m_{N-1})
\end{equation}
and
\begin{equation} \label{relcharge}
\charge(N)-\charge(\bar N)
= {\charge}(m_N) - {\charge}(m_{N-1})
\end{equation}
where for simplicity we denote the charge of the letter $N$ in
$T$ by $\charge(N)$, and
the charge of the letter $N$ in
$U$ by $\charge(\bar N)$ (and similarly for
cocharge).

The next lemma shows that both problems are equivalent.
\begin{lemma} We have
\begin{equation}
\cocharge(N)-\cocharge(\bar N)= {\cocharge}(m_N) - {\cocharge}(m_{N-1})
\iff
\charge(N)-\charge(\bar N)= {\charge}(m_N) - {\charge}(m_{N-1})
\end{equation}
\end{lemma}
\begin{proof}  Suppose that $T$ and $U$ differ by
a row move.  We have to show that
\begin{equation} \label{aprouver1}
\cocharge(N)-\cocharge(\bar N)= |m_N| - |m_{N-1}|
\iff
\charge(N)=\charge(\bar N)
\end{equation}
>From  Lemma~\ref{lemmechargecochargen}, we get
\begin{equation}
\cocharge(N)-\cocharge(\bar N)=(N-\charge(N)-|c_N|)-(N-\charge(\bar N)-|\tilde c_N|)=
\charge(\bar N)+|\tilde c_N|-\charge(N)-|c_N|
\end{equation}
which leads to
\begin{equation} \label{eqnnp}
\cocharge(N)-\cocharge(\bar N)= |m_N| - |m_{N-1}| \iff
\charge(\bar N)+|\tilde c_N|-\charge(N)-|c_N|=  |m_N| - |m_{N-1}|
\end{equation}
By inspection of Step 1 and Step 3 (in the row case), it is easy to deduce
that in all cases
\begin{equation}
 |\tilde c_{N}|- |c_N| =
 |m_N| - |m_{N-1}|
\end{equation}
and thus \eqref{aprouver1} follows from \eqref{eqnnp}.
The proof when $T$ and $U$ differ by a column move
is identical.
\end{proof}
It thus suffices to prove \eqref{relcocharge} when $T$ and $U$ differ by a
column move, and \eqref{relcharge} when $T$ and $U$ differ by a row move.
Observe that in both cases the right hand side of the equations is then equal to zero.
We will proceed by induction.  We will suppose that the relations hold for all
$N$ up to $N=n$ and then show that the case $N=n+1$ also holds, that is, that
\begin{enumerate}
\item[$\bullet$] $\cocharge(n+1)=\cocharge(\bar n+1)$ when  $T$ and $U$ differ by a
column move
\item[$\bullet$]
$\charge(n+1)=\charge(\bar n+1)$
when $T$ and $U$ differ by a
row move
\end{enumerate}
The proof will rely on commutative diagrams of the type
\begin{eqnarray}
\begin{CD}
T_{n} @ > {m_n } >> U_{n}   \\
@VV{c_{n+1}}V @VV{\tilde{c}_{n+1}}V
\\
T_{n+1} @ > {m_{n+1}} >>  U_{n+1}
\end{CD}
\end{eqnarray}
 where $T_n$ and $U_n$ denote standard $k$-shape tableaux of $n$ letters,
and where, for simplicity, we use $T_{n+1}$ instead of ${\rm sh}(T_{n+1})$.
We will also keep denoting
by $\bar n$ the letter $n$ in $U$.  For instance $\bar n^\uparrow$
denotes the highest occurrence of the letter $n$ in $U$, while
$n^\uparrow$
denotes the highest occurrence of the letter $n$ in $T$.

We now proceed to analyze all the possible cases.

\subsection{Row and column maximization}
We first consider the situation where $c_{n+1}$ is not maximal.
The maximization described in Step 1 and 2 is such that
\begin{eqnarray}
\begin{CD}
T_{n} @ > {\emptyset } >> T_{n}   \\
@VV{c_{n+1}}V @VV{\tilde{c}_{n+1}}V
\\
T_{n+1} @ > {m } >>  U_{n+1}
\end{CD}
\end{eqnarray}
where $m$ is a maximization below or above of the cover $c_{n+1}$.
Suppose that $m$ is a row move (maximization below).  We have that $n_+^\uparrow$ and $\bar n_+^\uparrow$ lie in the same position (above the move),
as do obviously $n^\uparrow$ and $\bar n^\uparrow$ (given that $T_n=U_n$).
Since $\charge(n+1)$ and $\charge(\bar n+1)$ are computed using the $k$-shape
corresponding to the shape of $T_n$, it is immediate that
$\charge(n+1)-\charge(n)=\charge(\bar n+1)-\charge(\bar n)=\charge(\bar n+1)-\charge(n)$.
Therefore, we get $\charge(n+1)=\charge(\bar n+1)$.

If $m$ is a column move, we can use a similar analysis (this
time $n_+^\downarrow$ and $\bar n_+^\downarrow$ lie in the same position below the move), to show that  $\cocharge(n+1)=\cocharge(\bar n+1)$.

\subsection{Maximal pushout (row case)}

We now show that the charge is conserved in the pushout of a maximal
cover $c_{n+1}$ and a row move $m$.
\begin{lemma} \label{lemmarowpush}
Consider the following situation
\begin{eqnarray}\label{diagcolum2bis2}
\begin{CD}
T_{n-1} @ > {m'} >> U_{n-1}   \\
@VV{c_n}V @VV{\tilde{c}_{n}}V
\\
T_{n} @ > {{m}} >>  U_{n} \\
@VV{c_{n+1}}V @VV{\tilde{c}_{n+1}}V
\\
T_{n+1} @ > {\tilde{m}} >>  U_{n+1} \\
\end{CD}
\end{eqnarray}
where $c_{n+1}$ is maximal, and where either
$c_n$ is maximal or $m$ is
is the maximization below of $c_n$ (in which case $m'$ is empty).
Then  $\charge(n+1)=\charge(\bar n+1)$
if $\charge(n)=\charge(\bar n)$.
\end{lemma}

Before proceeding to the proof of the lemma, we need to establish a few elementary results.
Let $m$ be a row move from  $\lambda$ to $\mu$ originating, as in Lemma~\ref{lemmarowpush}, either from a maximization below of $c_{n}$ or from the maximal
pushout of  the pair $(c_n,m')$.   The following observations follow easily
from the definition
of maximal pushout in the row case.

(i) $n^\uparrow$ and $\bar n^\uparrow$ are in the same position

(ii)  $n_+^\uparrow$ and $\bar n_+^\uparrow$ are in the same position

(iii) $n^\uparrow$ is never in a row that belongs to the move $m$.

(iv)  $n_+^\uparrow$ is never in a row that belongs to the move $m$.

We also need to describe how the move $m$ affects $k$-connectedness
between rows.
We will always consider that
$r_1$ and $r_2$ ($r_1 >r_2$) are two $k$-connected rows in $\lambda$.  There are essentially 3 cases to consider.

\begin{enumerate}
\item  If the move $m$ intersects row $r_2$ but does not continue
above, then as seen in the picture, the row that corresponds to the negatively
modified column of $m$ is now $k$-connected to row $r_2+1$, while
the other rows remain connected in the same way:

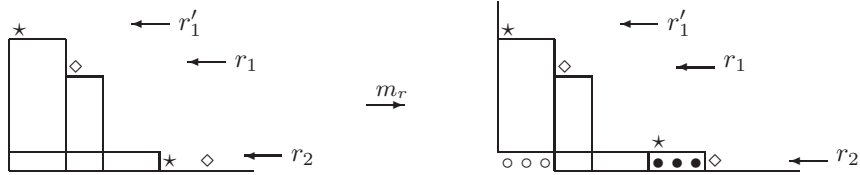
\setlength{\unitlength}{0.25mm}
    \begin{picture}(430,100)(0,0)

 \put(12,82){$\star$}
 \put(92,12){$\star$}
 \put(42,62){$\diamond$}
 \put(112,12){$\diamond$}

 \put(272,82){$\star$}
 \put(352,22){$\star$}
 \put(302,62){$\diamond$}
 \put(382,12){$\diamond$}

 \put(10,10){\line(1,0){130}}  \put(10,10){\line(0,1){70}}
 \put(10,20){\line(1,0){80}}  \put(40,10){\line(0,1){70}}           
 \put(40,60){\line(1,0){20}}   \put(60,10){\line(0,1){50}}
 \put(10,80){\line(1,0){30}} \put(90,10){\line(0,1){10}}

\put(100,85){$r'_1$}
\put(130,65){$r_1$}
\put(160,15){$r_2$}

\put(95,88){\vector(-1,0){20}}
\put(125,68){\vector(-1,0){20}}
\put(155,18){\vector(-1,0){20}}

\put(200,45){\vector(1,0){20}}
\put(205,50){\small{$m_r$}}

 \put(300,10){\line(1,0){130}}  \put(270,20){\line(0,1){80}}
 \put(270,20){\line(1,0){110}}  \put(300,10){\line(0,1){70}}           
 \put(300,60){\line(1,0){20}}   \put(320,10){\line(0,1){50}}
 \put(270,80){\line(1,0){30}} \put(350,10){\line(0,1){10}}
 \put(380,10){\line(0,1){10}}

\put(352,11){\small $\bullet$}
\put(362,11){\small $\bullet$}
\put(372,11){\small $\bullet$}

\put(272,11){\small$\circ$}
\put(282,11){\small$\circ$}
\put(292,11){\small$\circ$}

\put(360,85){$r'_1$}
\put(390,65){$r_1$}
\put(450,15){$r_2$}

\put(355,88){\vector(-1,0){20}}
\put(385,65){\vector(-1,0){20}}
\put(445,15){\vector(-1,0){20}}

\end{picture}

\item  If both rows $r_1$ and $r_2$ belong to $m$, then rows $r_1+1$ and
$r_2+1$ are $k$-connected in $\mu$, while row $r_1$ and $r_2$ remain connected
(if there is still an addable corner in  row $r_1$ of $\mu$)

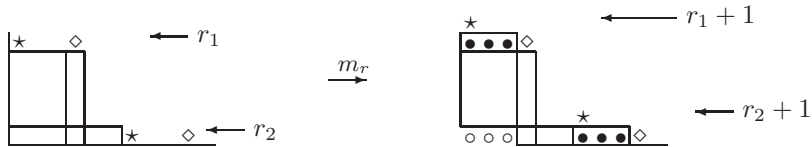
\setlength{\unitlength}{0.25mm}
    \begin{picture}(350,90)(0,0)

 \put(12,62){$\star$}
 \put(72,12){$\star$}
 \put(252,72){$\star$}
 \put(312,22){$\star$}

 \put(42,62){$\diamond$}
 \put(102,12){$\diamond$}
 \put(282,62){$\diamond$}
 \put(342,12){$\diamond$}

\put(10,10){\line(1,0){110}}  \put(10,10){\line(0,1){60}}
\put(10,20){\line(1,0){60}}  \put(40,10){\line(0,1){50}}             
\put(10,60){\line(1,0){40}}  \put(70,10){\line(0,1){10}}
\put(50,10){\line(0,1){50}}

\put(110,65){$r_1$}
\put(140,15){$r_2$}

\put(105,68){\vector(-1,0){20}}
\put(135,18){\vector(-1,0){20}}

\put(180,45){\vector(1,0){20}}
\put(185,50){\small{$m_r$}}

\put(280,10){\line(1,0){80}}  \put(250,20){\line(0,1){50}}
\put(250,70){\line(1,0){30}}  \put(280,60){\line(0,1){10}}
\put(250,20){\line(1,0){90}}  \put(280,10){\line(0,1){50}}             
\put(250,60){\line(1,0){40}}  \put(340,10){\line(0,1){10}}
\put(310,10){\line(0,1){10}}
\put(290,10){\line(0,1){50}}

\put(312,11){\small $\bullet$}
\put(322,11){\small $\bullet$}
\put(332,11){\small $\bullet$}

\put(252,61){\small$\bullet$}
\put(262,61){\small$\bullet$}
\put(272,61){\small$\bullet$}

\put(252,11){\small$\circ$}
\put(262,11){\small$\circ$}
\put(272,11){\small$\circ$}

\put(370,75){$r_1+1$}
\put(400,25){$r_2+1$}

\put(365,78){\vector(-1,0){40}}
\put(395,28){\vector(-1,0){20}}

\end{picture}

\item  If row $r_1$ belongs to the move $m$ but $r_2$ does not then
$r_1+1$ now connects with $r_2$ in $\mu$, while row $r_1$ and $r_2$ remain connected (if there is still an addable corner in  row $r_1$ of $\mu$)

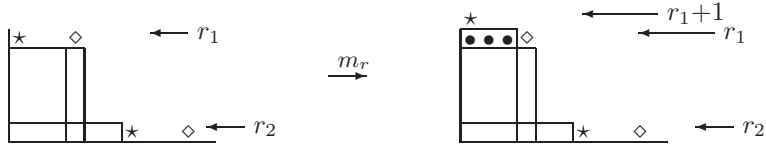
\setlength{\unitlength}{0.25mm}
    \begin{picture}(350,90)(0,0)

 \put(12,62){$\star$}
 \put(72,12){$\star$}
 \put(252,72){$\star$}
 \put(312,12){$\star$}

 \put(42,62){$\diamond$}
 \put(102,12){$\diamond$}
 \put(282,62){$\diamond$}
 \put(342,12){$\diamond$}

\put(10,10){\line(1,0){110}}  \put(10,10){\line(0,1){60}}
\put(10,20){\line(1,0){60}}  \put(40,10){\line(0,1){50}}             
\put(10,60){\line(1,0){40}}  \put(70,10){\line(0,1){10}}
\put(50,10){\line(0,1){50}}

\put(110,65){$r_1$}
\put(140,15){$r_2$}

\put(105,68){\vector(-1,0){20}}
\put(135,18){\vector(-1,0){20}}

\put(180,45){\vector(1,0){20}}
\put(185,50){\small{$m_r$}}

\put(250,10){\line(1,0){110}}  \put(250,10){\line(0,1){60}}
\put(250,70){\line(1,0){30}}  \put(280,60){\line(0,1){10}}
\put(250,20){\line(1,0){60}}  \put(280,10){\line(0,1){50}}             
\put(250,60){\line(1,0){40}} \put(310,10){\line(0,1){10}}
\put(290,10){\line(0,1){50}}

\put(252,61){\small$\bullet$}
\put(262,61){\small$\bullet$}
\put(272,61){\small$\bullet$}

\put(390,65){$r_1$}
\put(360,75){$r_1$+1}
\put(400,15){$r_2$}

\put(385,68){\vector(-1,0){40}}
\put(355,78){\vector(-1,0){40}}
\put(395,18){\vector(-1,0){20}}

\end{picture}

\end{enumerate}

We are now in a position to prove Lemma~\ref{lemmarowpush}.
\begin{proof}[Proof of Lemma~\ref{lemmarowpush}]
 Suppose that
$n^\uparrow$ is above $n_+^\uparrow$.
Let $r_1$
be the row above that of $n^\uparrow$ in $\lambda$, and let
$r_1,\dots,r_{\ell+1}$ be the connected rows below row $r_1$ such
that $r_{\ell} >s$ and $r_{\ell+1} \leq s$, where $s$ is the row of
$n_+^\uparrow$ in $\lambda$.  From observations (i) and (ii),  in $\mu$ row $r_1$ is still
the row above that of $\npup$ and $s$ is still the row of
$\bar n_+^\uparrow$.  From the analysis of the
3 cases considered above, in $\mu$
the connected rows below
$r_1$ will be $r_1,\dots,r_{i-1},r_i+1,\dots,r_{j}+1,r_{j+1},\dots,r_{\ell+1}$
if rows $r_i$ up to $r_j$ belong to the move $m$. Note that $i>1$ since,
as we just mentioned,
$r_1$ did not change.   We have to show that
$[r_1,s)_k$ is the same in $\lambda$ and in $\mu$.  In $\lambda$, we
have $[r_1,s)_k=\ell$.  In $\mu$ we will still have  $[r_1,s)_k=\ell$ unless
$r_{\ell+1}=s$ and $s$ is a row of the move.
But this is impossible from (iv).

Now suppose that
$\nup$ is weakly below $\nplusup$.
Let $r_1$
be the row of $\nplusup$ in $\lambda$, and let
$r_1,\dots,r_{\ell+1}$ be the connected rows below row $r_1$ such
that $r_{\ell} \geq s$ and $r_{\ell+1} < s$, where $s$ is the row above that
of $\nup$ in $\lambda$.   From observation (i) and (ii), in $\mu$ row $r_1$ is still
the row of $\npluspup$ and $s$ is still the row above that of
$\npup$.   In $\mu$, the connected rows below
$r_1$ will again be $r_1,\dots,r_{i-1},r_i+1,\dots,r_{j}+1,r_{j+1},
\dots,r_{\ell+1}$, with $i>1$,
if rows $r_i$ up to $r_j$ belong to the move $m$.  We have to show that
$(r_1,s]_k$ is the same in $\lambda$ and in $\mu$.  In $\lambda$, we
have $(r_1,s]_k=\ell-1$.  In $\mu$ we will still have  $(r_1,s]_k=\ell-1$ unless
$r_{\ell+1}=s-1$ and $s-1$ is a row of the move.
But this is impossible by (iii) since $s-1$ is the row of $\nup$.
\end{proof}

\subsection{Maximal pushout (column case)}
We need to show that the cocharge is conserved in the pushout of a maximal
cover $c_{n+1}$ and a column move $m$.
\begin{lemma}\label{lemmacolpush}
 Consider the following situation
\begin{eqnarray}
\begin{CD}
T_{n-1} @ > {m'} >> U_{n-1}   \\
@VV{c_n}V @VV{\tilde{c}_{n}}V
\\
T_{n} @ > {{m}} >>  U_{n} \\
@VV{c_{n+1}}V @VV{\tilde{c}_{n+1}}V
\\
T_{n+1} @ > {\tilde{m}} >>  U_{n+1} \\
\end{CD}
\end{eqnarray}
where $c_{n+1}$ is maximal, and where either
$c_n$ is maximal or $m$ is
is the maximization above of $c_n$ (in which case $m'$ is empty).
Then  $\cocharge(n+1)=\cocharge(\bar n+1)$
if $\cocharge(n)=\cocharge(\bar n)$.
\end{lemma}

Before proceeding to the proof of the lemma, we need as in the previous subsection to establish a few elementary results.
Let $m$ be a row move from  $\lambda$ to $\mu$ originating, as in Lemma~\ref{lemmacolpush}, either from a maximization above of $c_{n}$ or from the maximal
pushout of  the pair $(c_n,m')$.   The following observations follow easily
from the definition
of maximal pushout in the column case.

(i) $\ndown$ and $\npdown$ are in the same position

(ii)  $\nplusdown$ and $\npluspdown$ are in the same position

(iii) $\ndown$ is never in a row that belongs to the move $m$.

(iv)  $\nplusdown$ is never in a row that belongs to the move $m$.

We also need to describe how the move $m$ affects $k$-connectedness
between rows.  The main claim is the following.
\begin{lemma}\label{lemmacoladd}
 Let $m$ be a column move from $\lambda$ to $\mu$.
Suppose that $\lambda$ and $\mu$ have an addable corner in row $r_1$.  Then
the $k$-connected row $r_2$ below row $r_1$ is the same in $\lambda$ and in $\mu$.
\end{lemma}

\begin{proof} We consider all possible cases and see that the result
always hold.
\begin{enumerate}
\item If neither  $r_1$ nor $r_2$  belong to the move $m$,
then the result is immediate.
\item If only one of $r_1$ and $r_2$ belong to $m$, we have either

    \setlength{\unitlength}{0.25mm}
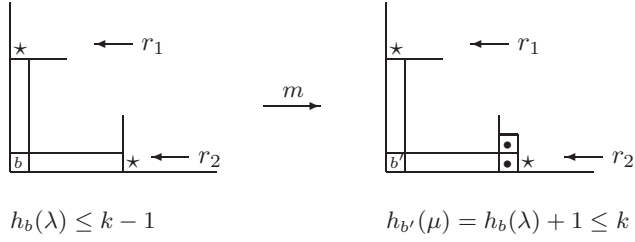
    \begin{picture}(330,150)(0,-10)

 \put(10,30){\line(1,0){110}}  \put(10,30){\line(0,1){90}}
 \put(10,40){\line(1,0){60}}  \put(20,30){\line(0,1){60}}           
 \put(10,90){\line(1,0){30}}   \put(70,30){\line(0,1){30}}

\put(80,95){$r_1$}
\put(110,35){$r_2$}
\put(12,32){\tiny $b$}

\put(75,98){\vector(-1,0){20}}
\put(105,38){\vector(-1,0){20}}

\put(210,30){\line(1,0){110}}  \put(210,30){\line(0,1){90}}
 \put(210,40){\line(1,0){60}}  \put(220,30){\line(0,1){60}}           
 \put(210,90){\line(1,0){30}}   \put(270,30){\line(0,1){30}}

 \put(270,40){\line(1,0){10}}   \put(280,30){\line(0,1){20}}
 \put(270,50){\line(1,0){10}}

\put(280,95){$r_1$}
\put(330,35){$r_2$}
\put(212,32){\tiny$b'$}

\put(275,98){\vector(-1,0){20}}
\put(325,38){\vector(-1,0){20}}

\put(272,32){\tiny $\bullet$}
\put(272,42){\tiny $\bullet$}

\put(145,65){\vector(1,0){30}}
\put(155,70){\small{$m$}}

\put(12,92){$\star$}
\put(72,32){$\star$}
\put(212,92){$\star$}
\put(282,32){$\star$}

\put(10,0){\small $h_b(\lambda) \leq k-1$} \put(210,0){\small $h_{b'}(\mu) =
h_b(\lambda)+1 \leq k$}

\end{picture}

or

        \setlength{\unitlength}{0.25mm}
    
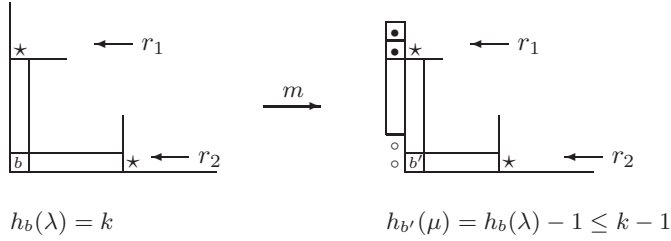
\begin{picture}(330,150)(0,-10)

 \put(10,30){\line(1,0){110}}  \put(10,30){\line(0,1){90}}
 \put(10,40){\line(1,0){60}}  \put(20,30){\line(0,1){60}}           
 \put(10,90){\line(1,0){30}}   \put(70,30){\line(0,1){30}}

\put(80,95){$r_1$}
\put(110,35){$r_2$}
\put(12,32){\tiny $b$}

\put(212,92){\tiny $\bullet$}
\put(212,102){\tiny $\bullet$}

\put(212,32){\tiny $\circ$}
\put(212,42){\tiny $\circ$}

\put(75,98){\vector(-1,0){20}}
\put(105,38){\vector(-1,0){20}}

\put(220,30){\line(1,0){100}}  \put(210,50){\line(0,1){60}}
 \put(210,50){\line(1,0){10}}  \put(220,30){\line(0,1){80}}           
 \put(210,90){\line(1,0){30}}   \put(270,30){\line(0,1){30}}

\put(230,30){\line(0,1){60}}   \put(210,100){\line(1,0){10}}
\put(220,40){\line(1,0){50}}   \put(210,110){\line(1,0){10}}

\put(280,95){$r_1$}
\put(330,35){$r_2$}
\put(222,32){\tiny $b'$}

\put(275,98){\vector(-1,0){20}}
\put(325,38){\vector(-1,0){20}}

\put(145,65){\vector(1,0){30}}
\put(155,70){\small{$m$}}

\put(12,92){$\star$}
\put(72,32){$\star$}
\put(222,92){$\star$}
\put(272,32){$\star$}

\put(10,0){\small $h_b(\lambda) =k$} \put(210,0){\small $h_{b'}(\mu) = h_b(\lambda)-1 \leq k-1$}

\end{picture}

\item If $r_1$ and $r_2$ both belong to $m$

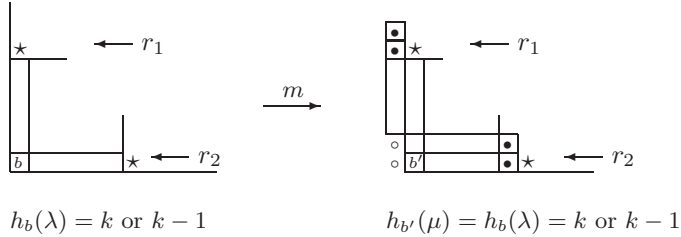
        \setlength{\unitlength}{0.25mm}
    \begin{picture}(330,150)(0,-10)

 \put(10,30){\line(1,0){110}}  \put(10,30){\line(0,1){90}}
 \put(10,40){\line(1,0){60}}  \put(20,30){\line(0,1){60}}           
 \put(10,90){\line(1,0){30}}   \put(70,30){\line(0,1){30}}

\put(80,95){$r_1$}
\put(110,35){$r_2$}
\put(12,32){\tiny $b$}

\put(212,32){\tiny $\circ$}
\put(212,42){\tiny $\circ$}

\put(75,98){\vector(-1,0){20}}
\put(105,38){\vector(-1,0){20}}

\put(220,30){\line(1,0){100}}  \put(210,50){\line(0,1){60}}
 \put(210,50){\line(1,0){60}}  \put(220,30){\line(0,1){80}}           
 \put(210,90){\line(1,0){30}}   \put(270,30){\line(0,1){30}}

 \put(270,40){\line(1,0){10}}   \put(280,30){\line(0,1){20}}
 \put(270,50){\line(1,0){10}}

\put(230,30){\line(0,1){60}}   \put(210,100){\line(1,0){10}}
\put(220,40){\line(1,0){50}}   \put(210,110){\line(1,0){10}}

\put(280,95){$r_1$}
\put(330,35){$r_2$}
\put(222,32){\tiny$b'$}

\put(275,98){\vector(-1,0){20}}
\put(325,38){\vector(-1,0){20}}

\put(145,65){\vector(1,0){30}}
\put(155,70){\small{$m$}}

\put(12,92){$\star$}
\put(72,32){$\star$}
\put(222,92){$\star$}
\put(282,32){$\star$}

\put(212,92){\tiny $\bullet$}
\put(212,102){\tiny $\bullet$}

\put(272,32){\tiny $\bullet$}
\put(272,42){\tiny $\bullet$}

\put(10,0){\small $h_b(\lambda) =k$ or $k-1$} \put(210,0){\small $h_{b'}(\mu) = h_b(\lambda)=k$ or $k-1$}
\end{picture}
\end{enumerate}

\end{proof}

We now proceed to the proof of Lemma~\ref{lemmacolpush}.

\begin{proof}[Proof of Lemma~\ref{lemmacolpush}]
 Suppose that
$\ndown$ is above $\nplusdown$.    By (i) and (ii),
we have that $\ndown$ coincides with $\npdown$ and that
$\nplusdown$ coincides with $\npluspdown$.  The row $r$
above that of $\ndown$ still has an addable corner in $\mu$
from (i). Using Lemma~\ref{lemmacoladd} again and again we get that
the string of $k$-connected rows below $r$ is the same in $\lambda$ and $\mu$.
It is thus immediate that $\cocharge(\bar n+1)=\cocharge(n+1)$ if
$\cocharge(\bar n)=\cocharge(n)$.

Suppose that
$\ndown$ is above $\nplusdown$.    By (i) and (ii),
we still have that $\ndown$ coincides with $\npdown$ and that
$\nplusdown$ coincides with $\npluspdown$.
The row $r$ of $\nplusdown$ thus has an addable corner in $\lambda$ and
$\mu$ by definition. Using Lemma~\ref{lemmacoladd} again and again we get that
the string of $k$-connected rows below $r$ is the same in $\lambda$ and $\mu$.
It is then again  immediate that $\cocharge(\bar n+1)=\cocharge(n+1)$ if
$\cocharge(\bar n)=\cocharge(n)$.

\end{proof}

\section{Conclusion}

The compatibility between charge and the weak bijection was only
established in the standard case.  We discuss here briefly the
obstruction to extending this compatibility to
the semi-standard case.  In order to extend  the charge of a $k$-tableau of dominant weight given at the end of Section~\ref{Schargektab}
to arbitrary $k$-shape tableaux of dominant weight,
we would need a way to order the various occurrences of the letters in
the tableau (such as is done when computing the charge).  Finding this order is essentially equivalent to defining
a Lascoux-Sch\"utzenberger-type action
of the symmetric group on $k$-shape tableaux \cite{LSplaxique,LLT}
that would
extend that on $k$-tableaux defined at the end of Section~\ref{Schargektab}.
Unfortunately, we have been able to define such an action
only on maximal (and reverse-maximal) tableaux (see \cite{MEtesis}),
which does  not appear
sufficient to prove the compatiblity with the weak bijection.
One of the reasons why the extension to the non-maximal case is
non-trivial is that the number of $k$-shape tableaux of a given shape
and weight depends in general on the weight (the number does not
depend on the weight only after a notion of equivalence
classes on $k$-shape tableaux of a given weight and shape has been
defined).

Our ultimate goal would
be to show that our Lascoux-Sch\"utzenberger-type action
of the symmetric group on $k$-tableaux
commutes with the weak bijection.  To be more precise,
recall that the  Lascoux-Sch\"utzenberger action of the symmetric group on
words associates to every given permutation an operator $\sigma$ that permutes
the weight of a word (or tableau) according to the permutation.
We conjecture that the weak bijection is such that (using the language
of \eqref{bijn})
$$T \longleftrightarrow (T^{(k)},[\bp]) \quad \iff \quad
\sigma(T) \longleftrightarrow (\sigma(T^{(k)}),[\bp])$$
where we still denote by $\sigma$ the corresponding operator in the
Lascoux-Sch\"utzenberger action of the symmetric group on $k$-tableaux.
That is, the pushout algorithm appears to commute with the action of $\sigma$ (observe that $[\bp]$ is left unchanged):
\begin{eqnarray}\label{diag4}
\begin{CD}
T @ > {} >>  \sigma(T) \\
@VV{[\bp]}V @VV{[\bp]}V
\\
T^{(k)} @ > {} >>  \sigma(T^{(k)}) \\
\end{CD}
\end{eqnarray}
Apart from providing us a tool to demonstrate the compatibility of the
charge and the weak bijection in the non-standard case, proving the commutativity \eqref{diag4}
would imply that the pushout is only necessary in the standard case (which is technically much simpler than in the non-standard case).  In effect, a natural standardization  ${\rm Std}$ follows from the Lascoux-Sch\"utzenberger
action of the symmetric group on words.  The standardization
${\rm Std}$ would then immediately  commute with the pushout:
$$T \longleftrightarrow (T^{(k)},[\bp]) \quad \iff \quad
{\rm Std}(T)
 \longleftrightarrow ({\rm Std}(T^{(k)}),[\bp])$$
Since the standardization has
a left inverse ${\rm Std}^{-1}$,
to obtain the non-standard pushout $T \longleftrightarrow (T^{(k)},[\bp])$, one would simply need to compute  the standard pushout $U={\rm Std}(T)
 \longleftrightarrow (U^{(k)},[\bp])$ and then get $T^{(k)}$ from the relation
$T^{(k)} ={\rm Std}^{-1}\circ U^{(k)}$.


\begin{thebibliography}{99}



\bibitem{AB} S. Assaf and S. Billey, \emph{Affine dual equivalence and $k$-Schur functions}, arXiv:1201.2128v1.

\bibitem{DM} A.J. Dalal, and J. Morse, \emph{A $t$-generalization for Schubert representatives of the affine Grassmannian}, preprint.


\bibitem{H}
M. Haiman, \emph{Hilbert schemes, polygraphs, and the Macdonald
positivity conjecture}, J. Am. Math. Soc. {\bf{14}} (2001), 941--1006.

\bibitem{Lam:affstan} T. Lam,
\emph{Affine Stanley symmetric functions},
Amer. J. Math.  {\bf 128}  (2006),  no. 6, 1553--1586.

\bibitem{Lam:kSchur} T. Lam,
\emph{Schubert polynomials for the affine Grassmannian},
J. Amer. Math. Soc.  {\bf 21}  (2008),  no. 1, 259--281.


\bibitem{LLMS} T. Lam, L. Lapointe, J. Morse, and M. Shimozono,
\emph{Affine insertion and Pieri rules for the affine Grassmannian},
Mem. Amer. Math. Soc. {\bf 208} (2010), no. 977.

\bibitem{LLMS2} T. Lam, L. Lapointe, J. Morse, and M. Shimozono,
\emph{The poset of $k$-shapes and branching rules for $k$-Schur functions},
to appear in Memoirs of the  American Mathematical Society. DOI: http://dx.doi.org/10.1090/S0065-9266-2012-00655-1.

\bibitem{LLMSSZ} Thomas Lam, Luc Lapointe, Jennifer Morse, Mark Shimozono, Anne Schilling, and Mike
Zabrocki,  \emph{k-Schur functions and affine Schubert calculus},
http://www.math.ucdavis.edu/anne/FQ2012/book-released-September2012.pdf.



\bibitem{LLM} L. Lapointe, A. Lascoux, and J. Morse,
\emph{Tableau atoms and a new Macdonald positivity conjecture},
Duke Math. J.  {\bf 116}  (2003),  no. 1, 103--146.

\bibitem{LM:filtra} L. Lapointe and J. Morse,
\emph{Schur function analogs for a filtration of the symmetric
function space},
J. Combin. Theory Ser. A  {\bf 101}  (2003),  no. 2, 191--224.


\bibitem{LM:cores} L. Lapointe and J. Morse,
\emph{Tableaux on $k+1$-cores, reduced
words for  affine permutations, and $k$-Schur expansions},   J. Comb. Th. A
{\bf 112} (2005), 44--81.


\bibitem{LM:QC}
L. Lapointe and J. Morse,
\emph{Quantum cohomology and the $k$-Schur basis},
Trans. Amer. Math. Soc. {\bf 360} (2008), 2021--2040.



\bibitem{LLT}
 A. Lascoux, B. Leclerc and J.-Y. Thibon,
\emph{The Plactic Monoid}, in Algebraic Combinatorics on Words,
M. Lothaire,  Encylopedia of Mathematics and its Applications {\bf 90},
Cambridge University Press, 2002.



\bibitem{LScharge}
 A. Lascoux and  M.-P. Sch\"utzenberger, \emph{Sur une
conjecture de H.O. Foulkes}, C.R. Acad. Sc. Paris. {\bf 294} (1978), 323--324.
 \bibitem{LSplaxique} A. Lascoux and M.-P. Sch\"utzenberger, \emph{Le mono\" \i de
plaxique}, Quaderni della Ricerca scientifica {\bf 109} (1981), 129--156.
\bibitem{M} I.~G. Macdonald, {Symmetric Functions and Hall
Polynomials}, 2nd edition, Clarendon Press, Oxford, 1995.

\bibitem{MEtesis}
M.E. Pinto, Ph.D. thesis, Universidad de Talca, 2013.

\end{thebibliography}
\end{document}